\documentclass[11pt]{amsart}
\usepackage{geometry}                
\usepackage{graphicx}
\usepackage{amssymb}
\usepackage{epstopdf}
\usepackage{amsmath}
\usepackage{color}
\usepackage{hyperref}
\usepackage{pdfsync}

\usepackage[all]{xy}
\xyoption{matrix}
\xyoption{arrow}

\newtheorem{thm}{Theorem}[subsection]
\newtheorem{lem}[thm]{Lemma}
\newtheorem{cor}[thm]{Corollary}
\newtheorem{prop}[thm]{Proposition}

\theoremstyle{definition}
\newtheorem{defn}[thm]{Definition}
\newtheorem{eg}[thm]{Example}

\theoremstyle{remark}
\newtheorem{rem}[thm]{Remark}
 
\numberwithin{equation}{section}

\newcommand{\mat}[1]{\ensuremath{
\left[\begin{matrix}#1
\end{matrix}\right]
}}
\newcommand\ul{\underline}
\newcommand\undim{\underline\dim\,}
 
\newcommand{\vs}[1]{\vskip .#1 cm} 
\newcommand{\noi}{\noindent}

\newcommand{\then}{\Rightarrow}

\newcommand{\ifff}{\Leftrightarrow}

 \newcommand{\onto}{\twoheadrightarrow}

\newcommand{\psip}[1]{\pi(#1)}

\DeclareMathOperator{\coker}{coker}
\DeclareMathOperator{\Hom}{Hom}%
\DeclareMathOperator{\Ext}{Ext}%
\DeclareMathOperator{\sgn}{sgn}%
\newcommand{\field}[1]{\mathbb{#1}}
\newcommand{\ZZ}{\ensuremath{{\field{Z}}}}

\newcommand{\RR}{\ensuremath{{\field{R}}}}
\newcommand{\QQ}{\ensuremath{{\field{Q}}}}
\newcommand{\NN}{\ensuremath{{\field{N}}}}

\newcommand{\kk}{{{\rm\bf{k}}}}  

\newcommand{\commentout}[1]{}

\newcommand{\cD}{\ensuremath{{\mathcal{D}}}}

\newcommand{\cR}{\ensuremath{{\mathcal{R}}}}

\newcommand{\cT}{\ensuremath{{\mathcal{T}}}}

\newcommand\del{\partial}

\title{Periodic trees and semi-invariants}

\author{Kiyoshi Igusa}
\address{Department of Mathematics, Brandeis University, Waltham, MA 02454}\email{igusa@brandeis.edu}
 \thanks{The first author is supported by NSA Grant \#H98230-13-1-0247}

\author{Gordana Todorov}
\address{Department of Mathematics, Northeastern University, Boston, MA 02115}
\email{g.todorov@neu.edu}
\thanks{The second author is supported by NSF Grant \#DMS-1103813}

\author{Jerzy Weyman}
\address{Department of Mathematics, University of Connecticut, Storrs, CT 06269}
\email{jerzy.weyman@uconn.edu}
\thanks{The third author is supported by NSF Grant \#DMS-1400740 and by the Alexander von Humboldt Foundation.}

\subjclass[2010]{
16G20; 20F55}

\begin{document}

\begin{abstract} Periodic trees are combinatorial structures which are in bijection with cluster tilting objects in cluster categories of affine type $\widetilde{A}_{n-1}$. The internal edges of the tree encode the $c$-vectors corresponding to the cluster tilting object, as well as the weights of the virtual semi-invariants associated to the cluster tilting object. We also show a direct relationship between the position of the edges of the tree and whether the corresponding summands of the cluster tilting object are preprojective, preinjective or regular.
\end{abstract}

\keywords{cluster tilting objects, c-vectors, exchange matrix, representations of quivers, real Schur roots, virtual representations, stability conditions}

\maketitle



 
\section*{Introduction}

The goal of this paper is to show that isomorphism classes of infinite $n$-periodic trees with a fixed sign function $\varepsilon$ are in bijection with cluster tilting objects in the cluster category of a quiver $\widetilde A_{n-1}^\varepsilon$ of affine type $\widetilde A_{n-1}$ given by the same sign function $\varepsilon$. ($\varepsilon$ determines orientation of the quiver.) It was shown in \cite{Lutz} that there is a correspondence between binary trees and tilting objects of type $A_{n}$ with straight orientation. The arbitrary $A_n$ case is explained in \cite{IOs} where a counting argument is used since the two sets are finite with the same cardinality.
 
 In both the finite and infinite case, the periodic trees give a visualization of cluster tilting objects and their relation to semi-invariants and $c$-vectors. In the affine case, the distinction between preprojective, regular and preinjective summands of a cluster tilting object are reflected in the geometry of the infinite periodic tree.

We define periodic trees abstractly (Definition \ref{def:periodic tree}). We need: a positive integer $n$, a function $\varepsilon$, and a periodic poset structure on $\ZZ$. More precisely: For a given $n\ge2$ we start with an $n$-periodic surjective sign function $\varepsilon:\ZZ/n\onto \{+,-\}$. This function determines an orientation for a quiver of type $\widetilde{A}_{n-1}$: the elements of $\ZZ/n$ correspond to the arrows of the quiver and the sign function determines the direction of each arrow. This oriented quiver is denote by $\widetilde A_{n-1}^\varepsilon$.

The following theorems relate $n$-periodic trees and cluster tilting objects.

\begin{thm}[Theorem \ref{thm:correspondence between trees and cluster tilting objects}]\label{main}
There is a 1-1 correspondence between isomorphism classes of $n$-periodic trees with sign function $\varepsilon$ and cluster tilting objects of the cluster category of the quiver $\widetilde A_{n-1}^\varepsilon$.
\end{thm}

\begin{thm}[Corollary \ref{cor: components of cluster from topology of tree}]\label{cor: which summands are regular, etc}
When a cluster tilting object corresponds to a tree, the summands of that object correspond to the edges of the tree. The position of the edge determines whether the corresponding summand is preprojective, regular, preinjective or shifted projective. More precisely:
\begin{enumerate}
\item Edges corresponding to regular summands of the cluster tilting object are those which lie on branches of the tree and not on the unique doubly infinite path in the tree.
\item Edges which lie in the doubly infinite path of $\cT$ are preprojective if either the tree has positive slope or the tree has zero slope and the edge has positive slope.
\item All other edges correspond either to preinjective summands or shifted projective summands of the cluster tilting object.
\end{enumerate}
\end{thm}

The correspondence between periodic trees and cluster tilting objects is obtained as follows. We first show that each periodic tree $\cT$ admits an embedding into the plane $\RR^2$. Since such an embedding is determined by the $y$-coordinates of $n+1$ consecutive points, the space of all equivalent embeddings of $\cT$ into the plane is a convex open subset $\cR(\cT)$ of Euclidean space $\RR^{n+1}$. We define (Definition \ref{def: FRn+1 to Rn}) a particular linear map of $\RR^{n+1}$ onto $\RR^n$ and shows that this determines a unique cluster tilting object $M$ so that the image of $\cR(\cT)$ is equal to $\cR(M)$. (Theorem \ref{thm:correspondence between trees and cluster tilting objects}). The proof of the theorem uses the stability theorem from \cite{IOTW} for virtual semi-invariants. 

First, to each edge $\ell$ of a periodic tree we associate an edge vector $\beta$ which is determined by the endpoints of the edge and the sign of its slope. The edge vector is a real Schur root of the quiver $\widetilde A_{n-1}^\varepsilon$ and every real Schur root occurs in some tree. Each edge vector $\beta$ corresponds to a weight of a semi-invariant defined on a virtual representation space of $\widetilde A_{n-1}^\varepsilon$. In order to determine the real support $D(\beta)$ (Definition \ref{def: real support D(b)}, Theorem \ref{original def of virtual semi-invariant}, \cite{IOTW}) of the semi-invariants of weight $\beta$ we consider the points in the closure of the open region $\cR(\cT)$ and prove the following result.

\begin{thm} [Lemma \ref{second stability lemma}]
Let $\beta$ be an edge vector for the edge $\ell$ of a periodic tree $\cT$. The points in the closure of $\cR(\cT)$ corresponding to the limit points of the embeddings in which the edge $\ell$ becomes horizontal form a subset of the real support $D(\beta)$ of $\beta$ and $D(\beta)$ is the closure of the union of such limit points over all trees having edge vector $\pm\beta$.
\end{thm}

Finally, edge vectors of a periodic tree are related to $c$-vectors in the following way.

\begin{thm} [Theorem \ref{thm: edge vectors are negative c-vectors}]\label{c-vectors}
The edge vectors of $\cT$ are equal to the negatives of the $c$-vectors of the corresponding cluster tilting object.
\end{thm}\vs2

%


We now describe the contents of the paper. In Section \ref{sec1} we define periodic trees and analyze their combinatorial structure. This section explains only the combinatorics of periodic trees. The relation with representation theory is explained in section \ref{sec2}.

Periodic trees are special cases of periodic posets. In subsection \ref{ss1.1: periodic posets} we briefly review the definition of a periodic poset which is based on the concept of a ``cyclic poset'' used in \cite{IT11} to construct certain Frobenius categories whose stable categories are continuous cluster categories.

In subsection \ref{ss1.2: periodic trees} we give the definition of an $n$-periodic tree which is admissible with respect to a given periodic sign functions. We derive several important properties. The first, Proposition \ref{prop: existence of periodic morphisms}, states that the set $\cR(\cT)$ of all embeddings (called periodic morphisms) for a periodic tree $\cT$ is convex and nonempty. After a change of variables this region will be seen, in Theorem \ref{thm:correspondence between trees and cluster tilting objects}, to be the set of all positive linear combinations of the components of a uniquely determined cluster tilting object corresponding to $\cT$. This leads to the Classification Theorem \ref{cor: classification of trees}: periodic trees fall into three classes: those of positive, negative and zero slope. The next result of this subsection, Corollary \ref{monotonic cor} will be interpreted later as the stability conditions for virtual semi-invariants. Finally, Proposition \ref{cor:edge vectors are Schur roots} is a disguised version of the statement that the set of all edge vectors of all admissible periodic trees is equal to the set of all real Schur roots of the corresponding quiver.

In subsection \ref{ss1.3: leaves of T} we recall the definition of a leaf and characterize the leaves of $\cT$ in terms of a periodic morphism on $\cT$. In subsection \ref{ss1.4: maxima and minima} we define internal maxima and minima of a periodic tree $\cT$ and characterize these vertices in terms of any periodic morphism on $\cT$. Maxima and minima only occur in trees of slope zero. In subsection \ref{ss1.5: periodic trees from periodic morphisms} we show that every $n$-periodic function $\ZZ\to \RR$ which is a monomorphism (and thus necessarily of nonzero slope) gives a periodic morphism for a unique $n$ periodic tree. This implies that the union of the disjoint regions $\cR(\cT)$ for all $n$-periodic trees $\cT$ is open and dense in $\RR^{n+1}$.

Section \ref{sec2} uses the ``{edge vectors}'' of a periodic tree to prove Theorem \ref{main}. In subsection \ref{ss2.1: edge vectors} we define edge vectors and derive basic properties which characterize these vectors. The relation to representations of the quiver $\widetilde{A}_{n-1}^\varepsilon$ is explained in subsections \ref{ss2.2: representations} and \ref{ss2.3: semi-invariants and cluster tilting objects}. In particular, subsection \ref{ss2.3: semi-invariants and cluster tilting objects} contains one of the main results of this paper: the correspondence between $n$-periodic trees and cluster tilting objects in the cluster category of $\kk \widetilde{A}_{n-1}^\varepsilon$. The correspondence, given in Theorems \ref{virtual stability theorem} and \ref{thm:correspondence between trees and cluster tilting objects} also gives, in Corollary \ref{cor: components of cluster from topology of tree} (Theorem \ref{cor: which summands are regular, etc} above), a description of which summands of the cluster tilting object are regular, preprojective and preinjective depending on the geometry of the periodic tree. Subsection \ref{ss2.5: example} explains these theorems on the example given in Figure \ref{fig0a}.

In Section \ref{sec3} we prove Theorem \ref{c-vectors} that edge vectors are negative $c$-vectors. We review the definition of an exchange matrix and the $c$-vectors of a cluster tilting object in subsection \ref{ss3.1: exchange matrix and cluster tilting objects}. We give several equivalent formulations of the theorem in subsection \ref{ss3.2: statement of the theorem}. The remainder of the paper gives the proof which is given by verifying that the edge vectors of a periodic tree transform according to the mutation rules of Fomin and Zelevinsky.\vs2

The authors would like to thank Kent Orr for numerous lengthy discussions leading to the results of this paper. We are also grateful to Hugh Thomas for helping us understand the relationship between semi-invariants and $c$-vectors. Also, an ongoing dialogue with Thomas Br\"ustle has motived us to add a discussion of periodic trees of slope zero for use in future work. The first author acknowledges the support of National Security Agency Grant \#H98230-13-1-0247. The second author acknowledges the support of National Science Foundation Grant \#DMS-1103813.  The third author acknowledges the support of the Alexander von Humboldt Foundation,
and of National Science Foundation Grants \#DMS-0901185 and DMS-1400740.

 

\section{Periodic trees}\label{sec1}

We define $n$-periodic trees with sign function $\varepsilon$ to be the Hasse diagrams of certain periodic posets whose vertex set $\{p_i\}$ is indexed by integers $i\in\ZZ$. This section explains only the combinatorics of periodic trees. The relation with representation theory is explained in the next section. For example, we show that there is a periodic tree having an edge with endpoints $p_i$ and $p_j$ if and only if the vector $\beta_{ij}$, defined in subsection \ref{ss2.1: edge vectors}, is a real Schur root of the quiver $\widetilde{A}_{n-1}^\varepsilon$.


\subsection{Periodic posets}\label{ss1.1: periodic posets}

Periodic trees are special cases of periodic posets. We briefly go over basic definitions of a periodic poset and related notions.

\begin{defn}
Let $P=\{p_i\,|\,i\in \ZZ\}$ be any set indexed by the integers and let $n\ge2$.
\begin{enumerate}
\item By an \emph{$n$-periodic function} on $P$ we mean a mapping $\psi:P\to\RR$ so that $m=\psi(p_{i+n})-\psi(p_i)$ is independent of $i$. We say that $\psi$ has \emph{slope} $\frac mn$. Let $V_n(P)$ be the vector space of all $n$-periodic maps on $P$.
\item A {partial ordering} on $P$ will be called an \emph{$n$-periodic partial ordering} if
\[
	p_i<p_j\ifff p_{i+n}<p_{j+n}\,.
\]
\item The set $P$ with an $n$-periodic partial ordering will be called an \emph{$n$-periodic poset}.
\item Given an $n$-periodic poset $P$, we get an \emph{$n$-periodic ordering} on $\ZZ$ by $i<'j\ifff p_i<p_j$.
\end{enumerate}
\end{defn}

\begin{rem} For any set $P$ indexed by $\ZZ$, a linear isomorphism
$
	V_n(P)\cong \RR^{n+1}
$
is given by
\[
	\psi\mapsto (\psi(p_1),\cdots,\psi(p_n),m)
\]
where $m=\psi(p_n)-\psi(p_0)$.
\end{rem}

\begin{rem} The notion of an $n$-periodic poset is closely related to the notion of ``cyclic poset''. A finite cyclic poset \cite{IT11} is an $n$-periodic poset with the additional property that, for any $p_i,p_j\in P$ there is an integer $k$ so that $p_i<p_{j+kn}$. We do not assume this here.
\end{rem}



Recall that, for any poset $P$, the \emph{Hasse diagram} of $P$, if it exists, is a directed graph whose vertices are the elements of $P$ and whose edges $p\to q$ indicate child-parent relationships. Recall that, $p$ is a \emph{child} of $q$ and $q$ is a \emph{parent} of $p$ if $p<q$ and there are no elements $x\in P$ so that $p<x<q$. The graph with vertex set $P$ and these edges is the Hasse diagram for $P$ if, whenever $p<q$, there is a directed path in the graph from $p$ to $q$. We consider only those posets $P$ which have Hasse diagrams. (This excludes, for example, the 2-periodic poset in which all even integers, ordered in the usual way, are less than all odd integers.)

Consider an edge $\ell$ in the Hasse diagram of a periodic poset $P$. Let the endpoints of $\ell$ be $p_i,p_j$ with $i<j$. There are two cases.
\begin{enumerate}
\item $p_i<p_j$. In this case we say that $\ell$ has \emph{positive slope}, $p_j$ is a \emph{right parent} of $p_i$ and $p_i$ is a \emph{left child} of $p_j$.
\item $p_i>p_j$. Then we say $\ell$ has \emph{negative slope}, $p_j$ is a \emph{right child} of $p_i$ and $p_i$ is a \emph{left parent} of $p_j$.
\end{enumerate}

\begin{defn}
Let $P=\{p_i\,|\, i\in\ZZ\}$ be an $n$-periodic poset and let $m\in\RR$. Then an \emph{$n$-periodic morphism} on $P$ is defined to be an $n$-periodic function $\psi:P\to \RR$ so that
\[
	 \psi(p_i)<\psi(p_j)\text{ if }p_i<p_j.
\]
Let $\cR(P)$ denote the set of all $\psi\in V_n(P)$ satisfying this condition. Then $\cR(P)$ is clearly a convex open subset of $V_n(P)$. To compare different periodic posets $P$ we will identity $V_n(P)$ with $V_n(\ZZ)$ using the correspondence $\psi\leftrightarrow \pi\in V_n(\ZZ)$ where $\pi(i)=\psi(p_i)$. The virtue of the notation $p_i$ is that $p_1>p_2$ makes more sense than $1>'2$.
\end{defn}


\subsection{Periodic trees}\label{ss1.2: periodic trees}

We will define periodic trees which are admissible with respect to a given periodic sign functions and derive several important properties. We show that every $n$-periodic tree admits an $n$-periodic morphism (\ref{prop: existence of periodic morphisms}) and use this to classify periodic trees into three classes: those of positive, negative and zero slope (\ref{cor: classification of trees}).

It will be convenient to consider partial ordering on subsets of $\ZZ$. We recall that a \emph{tree} is a connected graph with no cycles which we usually take to be infinite.

\begin{defn}\label{def: conditions T1234}
Let $S$ be a nonempty subset of $\ZZ$ and let $P=\{p_i\,|\, i\in S\}$ be a poset. Suppose that $P$ has a Hasse diagram $\cT$ which is a tree, i.e., simply connected. We say that the tree $\cT$ is \emph{admissible} with respect to a given sign function $\varepsilon:S\to\{+,-\}$ if it satisfies the following conditions. 
\begin{enumerate}
\item[T1.] Each $p_i\in P$ has at most one left parent, at most one right parent, at most one left child and at most one right child.
\item[T2.] If $\varepsilon_i=+$ then $p_i$ has at most one parent.
\item[T3.] If $\varepsilon_i=-$ then $p_i$ has at most one child.
\item[T4.] For any edge $\ell=(p_i,p_j)$ in $\cT$ and any $i<k<j$ in $S$ we have the following.
\begin{enumerate}
\item If $p_k<\min(p_i,p_j)$ then $\varepsilon_k=+$.
\item If $p_k>\max(p_i,p_j)$ then $\varepsilon_k=-$.
\end{enumerate}
\end{enumerate}
\end{defn}

An \emph{$n$-periodic sign function} is defined to be a mapping $\varepsilon:\ZZ\to\{+,-\},i\mapsto\varepsilon_i$ which is $n$-periodic, i.e., $\varepsilon_{i+n}=\varepsilon_i$ for all $i$. In the next section we will assume that $\varepsilon$ is surjective, i.e., $\varepsilon$ takes both values: $+,-$ and $n\ge2$. 

\begin{defn}\label{def:periodic tree}
An \emph{$n$-periodic tree} with periodic sign function $\varepsilon$ is defined to be a tree $\cT$ which is the Hasse diagram of an $n$-periodic poset $P$ which is admissible with respect to the $n$-periodic sign function $\varepsilon$. 
\end{defn}

We view the poset $P$ as the vertex set of the graph $\cT$ and we denote it by $P_\cT$. We also denote $\cR(P_\cT)$ by $\cR(\cT)$.

First we analyze the topology of an $n$-periodic tree. There is a free action of the additive group $\ZZ$ on $\cT$ given by $k\cdot p_i=p_{i+kn}$. Let $\cT/\ZZ$ denote the orbit space of this action.

\begin{lem}
$\cT/\ZZ$ is a connected graph with $n$ vertices and $n$ edges. It consists of one cycle with possible additional edges forming subtrees attached to this cycle at different points.
\end{lem}

\begin{proof}
Since $\cT$ is contractible, the orbit space $\cT/\ZZ$ is a $K(\ZZ,1)$, i.e., it is homotopy equivalent to a circle. Also, $\cT/\ZZ$ is a finite graph with $n$ vertices and at most $3n/2$ edges since each vertex is incident to at most three edges. However, the Euler characteristic of any space homotopy equivalent to a circle is zero. And the Euler characteristic of a finite graph is equal to the number of vertices minus the number of edges. Therefore, $\cT/\ZZ$ is a connected graph with $n$ vertices and $n$ edges. It follows that $\cT/\ZZ$ has exactly one cycle. Since $\cT/\ZZ$ is a $K(\ZZ,1)$, this cycle is a deformation retract of the entire graph. So, the rest of the graph must consist of trees, being attached to the cycle at single points. These points must be distinct since every vertex is incident to at most three edges.
\end{proof}

 In the universal covering $\cT$ of $\cT/\ZZ$ we conclude the following.
 
\begin{prop}
An $n$-periodic tree consists of a single periodic doubly infinite path with at most one branch (finite subtree) attached to each point. \qed
\end{prop}

The edges in the periodic infinite path are characterized by the property that removal of the edge breaks the tree into two infinite subtrees. The edges in the branches are characterized by the property that removal of the edge will break the tree into one finite subtree and one infinite subtree. There are $n$-periodic trees with no branches.

\begin{prop}\label{prop: existence of periodic morphisms}
Every $n$-periodic tree $\cT$ admits an $n$-periodic morphism $\psi:P_\cT\to\RR$. Therefore $\cR(\cT)$ is a nonempty convex open subset of $V_n(\ZZ)\cong\RR^{n+1}$.
\end{prop}

\begin{proof}
By the lemma, $\cT/\ZZ$ has exactly one cycle. There are two cases.

\emph{Case 1:} $\cT/\ZZ$ has no oriented cycles. Then the vertex set $P_\cT/\ZZ$ obtains a partial ordering and there is an order preserving monomorphism $P_\cT/\ZZ\to \ZZ$ (often called an ``admissible order'' on the vertices of an acyclic quiver) which gives, by composition, an $n$-periodic morphism $\psi:P_\cT\to P_\cT/\ZZ\to\ZZ\to\RR$ whose slope is, by definition, $\frac1n(\psi(p_{i+n})-\psi(p_i))=0$. So, the proposition holds in this case.

\emph{Case 2:} Now suppose that $\cT/\ZZ$ has an oriented cycle of length, say $k$. Then the periodic infinite path in $\cT$ is an oriented path $\gamma_\infty$. If $p_i$ is any vertex on this path then $p_{i+n}$ is also a vertex on $\gamma_\infty$. There are two subcases: the infinite oriented path goes from $p_i$ to $p_{i+n}$ making $p_i<p_{i+n}$ or it goes the other way making $p_i>p_{i+n}$.

\emph{Case 2a:} $p_i<p_{i+n}$. In this case, define a function $\psi$ on each node in the path $\gamma_\infty$ by $\psi(p_i)=0$ at the point $p_i$ on $\gamma_\infty$ and $\psi=1$ at the next point on $\gamma_\infty$ and so on until we reach $\psi(p_{i+n})=k$. This defines an $n$-periodic morphism of positive slope $\frac kn$ on the periodic infinite path $\gamma_\infty$ and there is no problem extending it to an $n$-periodic morphism on all of $P_\cT$.

\emph{Case 2b:} $p_i>p_{i+n}$. In this case we get an $n$-periodic morphism on $P_\cT$ of slope $-\frac kn$.
\end{proof}

The case-by-case analysis in the above proof gives the following.

\begin{thm}[Classification of periodic trees]\label{cor: classification of trees}
There are three types of $n$-periodic trees.
\begin{enumerate}
\item Trees with slope zero: $\cT/\ZZ$ has no oriented cycles. Equivalently, there exists an $n$-periodic morphism $P_\cT\to\RR$ with slope 0.
\item Trees with positive slope: $\cT$ has an infinite monotonically increasing path. Equivalently, every $n$-periodic morphism $P_\cT\to\RR$ has positive slope.
\item Trees with negative slope: $\cT$ has an infinite monotonically decreasing path. Equivalently, every $n$-periodic morphism $P_\cT\to\RR$ has negative slope.\qed
\end{enumerate}
\end{thm}

\begin{rem}\label{rem: zero slope trees have no oriented cycles}
For a tree with zero slope the sign function $\varepsilon$ is necessarily surjective since the infinite path in $\cT$ must have local maxima and minima.
\end{rem}

Examples of periodic trees with zero, positive and negative slopes are given in Figures \ref{fig0a},\ref{fig1},\ref{fig0c} respectively.

%
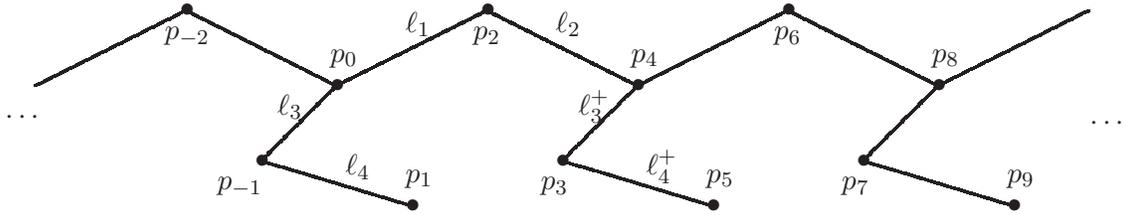
\begin{figure}[htbp]
\begin{center}
%
{
\setlength{\unitlength}{1cm}
{\mbox{
\begin{picture}(10,3)
      \thicklines
\put(0,0){
\qbezier(0,3)(1,2.5)(2,2)
\qbezier(2,2)(3,2.5)(4,3)
\qbezier(1,1)(1.5,1.5)(2,2)
\qbezier(1,1)(2,.7)(3,.4)
\put(-.08,-.08){
\put(1,1){$\bullet$}
\put(0,3){$\bullet$}
\put(-.2,2.7){$p_{-2}$}
\put(2,2.4){$p_0$}
\put(3,2.8){$\ell_1$}
\put(3,.8){$p_1$}
\put(2,2){$\bullet$}
\put(3,.4){$\bullet$}
\put(0.5,.7){$p_{-1}$}
\put(1.3,1.7){$\ell_3$}
\put(2.2,.9){$\ell_4$}
}
}
\put(-4,0){
\put(1.6,1.5){$\cdots$}
\qbezier(2,2)(3,2.5)(4,3)
}
\put(4,0){
\qbezier(0,3)(1,2.5)(2,2)
\qbezier(2,2)(3,2.5)(4,3)
\qbezier(1,1)(1.5,1.5)(2,2)
\qbezier(1,1)(2,.7)(3,.4)
\put(-.08,-.08){
\put(1,2.8){$\ell_2$}
\put(1.3,1.7){$\ell_3^+$}
\put(2.2,.9){$\ell_4^+$}
\put(1,1){$\bullet$}
\put(0,3){$\bullet$}\put(-.1,2.7){$p_{2}$}
\put(2,2.4){$p_4$}
\put(3,.8){$p_5$}
\put(2,2){$\bullet$}
\put(3,.4){$\bullet$}
\put(0.8,.7){$p_{3}$}
}
}
\put(8,0){
\qbezier(0,3)(1,2.5)(2,2)
\qbezier(2,2)(3,2.5)(4,3)
\qbezier(1,1)(1.5,1.5)(2,2)
\qbezier(1,1)(2,.7)(3,.4)
\put(-.08,-.08){
\put(1,1){$\bullet$}
\put(0,3){$\bullet$}\put(-.1,2.7){$p_{6}$}
\put(2,2.4){$p_8$}
\put(3,.8){$p_9$}
\put(2,2){$\bullet$}
\put(3,.4){$\bullet$}
\put(0.8,.7){$p_{7}$}
\put(4.1,1.5){$\cdots$}
}
}
\end{picture}}
}}
\caption{Example of periodic tree with slope zero (with $n=4$). The unique infinite path has edges of positive and negative slope ($\ell_1$ and $\ell_2$). The sign function is $\varepsilon_1=\varepsilon_2=\varepsilon_3=+$ and $\varepsilon_4=-$.}
\label{fig0a}
\end{center}
\end{figure}
%

%
\begin{figure}[htbp]
\begin{center}
%
{
\setlength{\unitlength}{1cm}
{\mbox{
\begin{picture}(10,2.5)
      \thicklines
\put(-2,1.25){
\qbezier(0,1.5)(2.5,1.25)(5,1)
}
\put(2,0){
\qbezier(0,1.5)(2.5,1.25)(5,1)
\qbezier(0,1.5)(1.5,1.75)(3,2)
\qbezier(1,2.25)(2,2.125)(3,2)
\qbezier(1,2.25)(1.5,2.5)(2,2.75)
\put(-.08,-.08){
\put(0,1.5){$\bullet$}
\put(-.1,1.2){$+$}
\put(1,2.5){$-$}
\put(2,3){$-$}
\put(3,2.2){$-$}
\put(1,2.25){$\bullet$}
\put(2,2.75){$\bullet$}
\put(3,2){$\bullet$}
}
}
\put(6,-1.25){
\qbezier(0,1.5)(2.5,1.25)(5,1)
\qbezier(0,1.5)(1.5,1.75)(3,2)
\qbezier(1,2.25)(2,2.125)(3,2)
\qbezier(1,2.25)(1.5,2.5)(2,2.75)
\put(-.08,-.08){
\put(-.1,1.2){$+$}
\put(1,2.5){$-$}
\put(2,3){$-$}
\put(3,2.2){$-$}
\put(0,1.5){$\bullet$}
\put(1,2.25){$\bullet$}
\put(2,2.75){$\bullet$}
\put(3,2){$\bullet$}
}
}
\end{picture}}
}}
\caption{Example of periodic tree with negative slope (with $n=4$). Vertices are labelled with just their sign $\varepsilon_i$. 
}
\label{fig0c}
\end{center}
\end{figure}
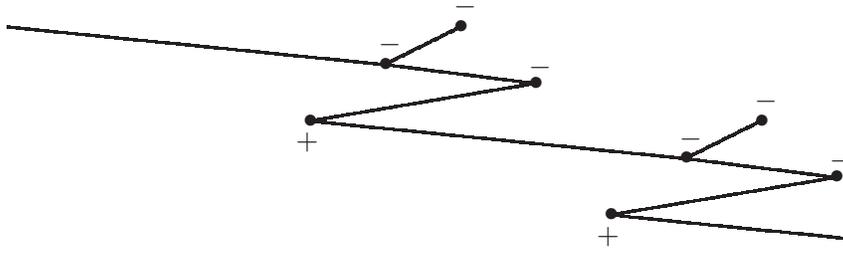

\begin{lem}\label{separation lemma}
Suppose that $\varepsilon_k=+$ (resp. $\varepsilon_k=-$) and $\lambda_L,\lambda_R$ are paths in $\cT$ starting at $p_k$ and passing through the left child and right child (resp. parents) of $p_k$. Then $\lambda_L$ stays to the left of $p_k$ and $\lambda_R$ stays to the right of $p_k$, i.e., $i\le k$ for every vertex $p_i$ in $\lambda_L$ and $k\le j$ for every vertex $p_j$ in $\lambda_R$.
\end{lem}


\begin{proof}
By symmetry we consider only $\lambda_R$ in the case $\varepsilon_k=+$. Let $m$ be the length of $\lambda_R$. If $m=1$ then $\lambda_R$ is a single edge connecting $p_k$ to its right child. So, $\lambda_R$ lies to the right of $p_k$. Now suppose $m\ge2$ and the lemma holds for paths of length $<m$.

\ul{Claim:} We may assume that $\lambda_R$ is monotonically decreasing.

Pf: If not it reaches a local max or min in its interior, say at $p_s$. Let $\lambda_R=\lambda_1\lambda_2$ where $\lambda_1$ is the part of $\lambda_R$ from $p_k$ to $p_s$. Then, by induction on $m$, $\lambda_1$ lies to the right of $p_k$. In particular, $k<s$. Also, replacing $p_k$ by $p_s$, we know by induction on $m$ that $\lambda_1,\lambda_2$ lie on opposite sides of $p_s$. So, $\lambda_2$ lies to the right of $p_s$. This would imply that $\lambda_R=\lambda_1\lambda_2$ lies entirely to the right of $p_k$ and we would be done. This proves the claim.

To finish the proof, let $\lambda_R=\lambda_0\ell$ where $\ell=(p_j,p_i)$ is the last edge in $\lambda_R$ and $p_j$ is the last vertex of $\lambda_0$. By induction on $m$ we know that $\lambda_0$ is entirely to the right of $p_k$. So, $k<j$. Since $\lambda_R$ is monotonically decreasing we have that $p_k>\max(p_i,p_j)$. Since $\varepsilon_k=+$ this implies that $k<i$ since, otherwise, $i<k<j$ violating Condition T4. So, $\lambda_R$ is to the right of $p_k$.
\end{proof}

We recall that, in any tree $\cT$, any two points of the tree, not necessarily vertices, are connected by a unique (minimal) path.

\begin{rem}\label{separation remark} This lemma can be rephrased in two ways.

(a) If a path $\lambda$ in $\cT$ reaches a local maximum or minimum at an internal point $p_k$ in $\lambda$ then the starting and ending points of $\lambda$ are on opposite sides of $p_k$. 

(b) If a path $\lambda$ in $\cT$ begins at a vertex $p_k$ and ends in a point in the interior of an edge $(p_i,p_j)$ where $i<k<j$ then $\lambda$ is monotonically increasing or decreasing. 

To see that (b) follows from (a), note that otherwise, $\lambda$ reaches a local max or min with $p_k$ on one side and $p_i,p_j$ on the other contradicting the assumption.
\end{rem}

\begin{thm}\label{thm:linear embedding of T}
Let $\cT$ be an $n$-periodic tree and $\psi:P_\cT\to\RR$ any $n$-periodic morphism.
\begin{enumerate}
\item The linear map $\overline\psi:\cT\to \RR^2$ given on vertices by $\overline\psi(p_k)=(k,\psi(p_k))$ is an embedding.
\item If $\varepsilon_k=+$ then $\overline\psi(\cT)$ is disjoint from the set of all $(k,y)$ where $y<\psi(p_k)$.
\item If $\varepsilon_k=-$ then $\overline\psi(\cT)$ is disjoint from the set of all $(k,y)$ where $y>\psi(p_k)$.
\end{enumerate}
\end{thm}

\begin{proof}
To prove (1) suppose that $p,q$ are two point in $\cT$ which map to the same point $z\in\RR^2$ under the linear mapping $\overline\psi:\cT\to\RR^2$. Let $\lambda$ be the unique path in $\cT$ connecting $p$ and $q$. The image of $\lambda$ in $\RR^2$ has points $v,w$ with maximal and minimal $y$-coordinates. Since $v\neq w$, one of them is not $z$, say $w\neq z$. Then $w$ is the image of some vertex $p_k$ which is a local minimum of the curve $\lambda$ and $p,q\neq p_k$. By Remark \ref{separation remark}(a), the endpoints $p,q$ of $\lambda$ lie on opposite sides of $p_k$ (in the $x$-direction). So, they cannot map to the same point in $\RR^2$.

To prove (2), suppose not. Then there is a point $q$ in $\cT$ so that $\overline \psi(q)$ lies directly below $\overline\psi(p_k)$ with $\varepsilon_k=+$. Then $q$ lies on an edge $\ell=(p_i,p_j)$ with $i<k<j$. Let $\lambda$ be the path in $\cT$ from $p_k$ to $q$. By Remark \ref{separation remark}(b), $\lambda$ must be monotonically decreasing. So, we have $p_k>\max(p_i,p_j)$ which contradicts condition T4.

The proof of (3) is analogous.
\end{proof}

Combining Theorem \ref{thm:linear embedding of T} and Remark \ref{separation remark}(b) we get the following important corollary. 

\begin{cor}\label{monotonic cor}
Suppose that $\ell=(p_i,p_j)$ is an edge in $\cT$ and $i<k<j$. Then the path in $\cT$ from $p_k$ to any point on $\ell$ is either monotonically increasing or monotonically decreasing depending on whether $\varepsilon_k=+$ or $-$ respectively. In particular, we have the following for any periodic morphism $\psi$ on $P_\cT$.
\begin{enumerate}
\item If $\varepsilon_k=+$ then $\psi(p_k)<\min(\psi(p_i),\psi(p_j))$.
\item If $\varepsilon_k=-$ then $\psi(p_k)>\max(\psi(p_i),\psi(p_j))$.\qed
\end{enumerate}
\end{cor}

\begin{rem}
We will see later (Proposition \ref{prop: psi satisfies stability conditions for D(b)}) that these are stability conditions for virtual semi-invariants. To completely analyze stability conditions for virtual semi-invariants, we need the following converse of the above corollary.
\end{rem}

\begin{cor}\label{converse of monotonic cor}
Suppose $i<j$ and $\psi(p_i),\psi(p_j)$ are consecutive elements of the image of $\psi:P_\cT\to\RR$. Suppose also that, for all $i<k<j$, conditions {(1)} and {(2)} in Corollary \ref{monotonic cor} hold. Then $\ell=(p_i,p_j)$ is an edge in $\cT$.
\end{cor}

\begin{proof}
Let $\lambda$ be the path in $\cT$ from $p_i$ to $p_j$. We claim that $\lambda$ stays between the horizontal lines $\RR\times \psi(p_i)$ and $\RR\times\psi(p_j)$ and therefore consists of a single edge from $p_i$ to $p_j$ no other points being in this region. To prove this claim, suppose not. Then $\lambda$ reaches either its highest or lowest point in its interior. If $\lambda$ reaches a minimum at $p_k$ then, by Remark \ref{separation remark}(a), we must have $\varepsilon_k=-$ and $i<k<j$ which contradicts \ref{monotonic cor}(2). If $\lambda$ reaches a maximum then \ref{monotonic cor}(1) is violated. So, the claim holds and the corollary follows.
\end{proof}

The above corollary will allow us to construct a periodic tree with any prescribed edge $(p_i,p_j)$ which satisfies (1) or (2) in the following proposition. We will see later that these conditions are equivalent to the statement that the edge vector corresponding to $\ell=(p_i,p_j)$ is a real Schur root. Namely, when $\varepsilon_i\neq \varepsilon_j$, the edge vector is either preprojective or preinjective and the edge vector is a regular exceptional root if and only if $\varepsilon_i=\varepsilon_j$ and $j-i<n$.

\begin{prop}\label{cor:edge vectors are Schur roots}
If $\ell=(p_i,p_j)$ with $i<j$ is an edge in an $n$-periodic tree $\cT$ and $\varepsilon$ is surjective then either 
$(1)$ $j-i<n$ or 
$(2)$ $\varepsilon_i\neq \varepsilon_j$.
In particular, the length $j-i$ of $\ell$ is not divisible by $n$.
\end{prop}

\begin{rem}\label{rem: first use of e surjective}
This is the only statement in this section which uses the assumption that $\varepsilon$ is surjective. When $\varepsilon$ is constant, $\cT$ can have an edge of any length $\le n$.
\end{rem}

\begin{proof} Suppose not. Then $j\ge i+n$ and $\varepsilon_i=\varepsilon_j$. Consider first the case $j>i+n$. Assume by symmetry that $\varepsilon_i=\varepsilon_j=\varepsilon_{i+n}=+$. Then $p_{i+n}<p_i,p_j$. Since $\ell'=(p_{i+n},p_{j+n})$ is also an edge of $\cT$ and $i+n<j<j+n$ we get $p_{j}<p_{i+n},p_{j+n}$ which is a contradiction.

So, $j=i+n$. Then $\ell$ and its translates form a straight line in the plane and this line cannot be connected to any point above it since $p_i$ can have only one parent. So, all other nodes of $\cT$ must be below this line. By Theorem \ref{thm:linear embedding of T}(3), we must have $\varepsilon_j=+$ for all $j$ contradicting the assumption that $\varepsilon$ is surjective.
\end{proof}


\subsection{Leaves of $\cT$}\label{ss1.3: leaves of T}

Recall that a \emph{leaf} of a graph is vertex which is an endpoint of exactly one edge. When a leaf and its incident edge are removed from any tree, what remains is a smaller tree. This is a very useful induction procedure. We will characterize the leaves of $\cT$ in terms of a periodic morphism $\psi$.

\begin{lem}\label{leaf corollary}
Suppose that $p_i$ is a leaf of an $n$-periodic tree $\cT$. Then:
\begin{enumerate}
\item If $\varepsilon_i=-$ then $p_i$ has one child and no parents.
\item If $\varepsilon_i=+$ then $p_i$ has one parent and no children. 
\end{enumerate}
\end{lem}

\begin{proof}
To prove (1) suppose that $\varepsilon_i=-$ and $p_i$ has a right parent $p_j$. Let $\lambda$ be the path in $\cT$ from $p_i$ to $p_{i-n}$. Then $\lambda$ goes first to $p_j$ which is on the right side of $p_i$ and ends on the left side of $p_i$. So, $\lambda$ goes through a point $q$ which is directly below $p_i$. By Remark \ref{separation remark}(b), the part of the path $\lambda$ from $p_i$ to $q$ is monotonically decreasing. But this is impossible since its first step is upward.

The other case of (1) and both cases of (2) are similar.
\end{proof}

\begin{prop}\label{second leaf corollary}
Suppose there are $i<j<k$, $\varepsilon_i=\varepsilon_j=\varepsilon_k=-$ and $\varepsilon_t=+$ for all $t\neq j$ between $i$ and $k$. Then, for any $n$-periodic morphism $\psi$ on $P_\cT$, the following are equivalent.
\begin{enumerate}
\item $p_j$ is a leaf of $\cT$.
\item $\psi(p_j)>\psi(p_s)$ for all $i\le s\le k$ with $s\neq j$.
\end{enumerate}
Furthermore, in that case, the unique child of $p_j$ is $p_s$ where $i\le s\le k$ has the property that $s\neq j$ and $\psi(p_s)>\psi(p_t)$ for all $i\le t\le k$, $t\neq j,s$.
\end{prop}

\begin{rem}\label{rem: dual leaf formula}
The dual statement also holds: when the signs are reversed we change the inequalities $\psi(p_j)>\psi(p_s)$ and $\psi(p_s)>\psi(p_t)$ to $\psi(p_j)<\psi(p_s)$ and $\psi(p_s)<\psi(p_t)$.
\end{rem}

\begin{proof}
$(2)\then(1)$ Assuming (2), $p_j$ cannot have parents since any edge starting at $\overline\psi(p_j)$ and going up would hit one of the two ascending walls above $\overline\psi(p_i)$ and $\overline\psi(p_k)$ before reaching any other node of $\overline\psi(\cT)$. Since $\varepsilon_j=-$, $p_j$ has at most one child. It is a leaf.

$(1)\then(2)$ Suppose that $p_j$ is a leaf of $\cT$ and $\psi(p_s)\ge \psi(p_j)$ for some $s\neq j$. Then it suffices to show that $s$ does not lie in the closed interval $[i,k]$. By the proposition above, $p_j$ has one child and no parent. Therefore the path from $\overline\psi(p_j)$ to $\overline\psi(p_s)$ starts by going down. So, it reaches a minimum at, say $\overline\psi(p_t)$. Then $\varepsilon_t=-$ and the vertical wall $t\times\RR$ cuts the path into two parts with $\overline\psi(p_j)$ and $\overline\psi(p_s)$ on opposite sides of the wall. Since $\varepsilon_t=-$, either $t\le i$ or $t\ge k$. In either case, $s$ must be outside the closed interval $[i,k]$.

For the last statement, take any $i\le s\le k$, $s\neq j$ with maximal $\psi(p_s)$. Then we will show that $p_s$ is the unique child of $p_j$. To do this, consider the path from $\overline\psi(p_j)$ to $\overline\psi(p_s)$. If this path is not monotonically descending then, by the same argument as in the previous paragraph, we can conclude that $s$ lies outside the closed interval $[i,k]$ which is a contradiction. Therefore, the path is monotonically decreasing. By maximality of $\psi(p_s)$, this path cannot pass through any other node in the closed interval $[i,k]$. It cannot go outside the interval since it cannot cross the vertical lines $i\times\RR$ and $k\times\RR$. Therefore, the path is a single edge and $p_s$ is the unique child of $p_j$ as claimed.
\end{proof}


\subsection{Maxima and minima}\label{ss1.4: maxima and minima}

By an \emph{internal maximum} of $\cT$ we mean a vertex with two children and no parents. Dually, an \emph{internal minimum} of $\cT$ is a vertex with two parents and no children. The key point about internal maxima and minima is that they create ``vertical walls'' (Proposition \ref{cor:max min3}(1)) which cuts the tree $\cT$ into pieces which are finite subtrees which are easier to analyze. We will characterize these points in terms of the $n$-periodic morphism $\psi$.

\begin{lem}\label{cor:max min1} 
Suppose that $p_i$ is a node of an $n$-periodic tree $\cT$. Then:
\begin{enumerate}
\item If $\varepsilon_i=+$ and $p_i$ has no parents then it has two children.
\item If $\varepsilon_i=-$ and $p_i$ has no children then it has two parents.
\end{enumerate}
\end{lem}

\begin{proof}
The other possibilities are excluded by Lemma \ref{leaf corollary}.
\end{proof}

\begin{lem}\label{cor:max min2}
Suppose there are $i<j<k$, $\varepsilon_i=\varepsilon_k=-$ and $\varepsilon_t=+$ for all $i<t<k$. In particular $\varepsilon_j=+$. Then, for any $n$-periodic morphism $\psi$ on $P_\cT$, the following are equivalent.
\begin{enumerate}
\item $p_j$ is an internal maximum of $\cT$.
\item $\psi(p_j)>\psi(p_s)$ for all $i\le s\le k$ with $s\neq j$.
\end{enumerate}
Furthermore, in that case, the left child of $p_j$ is $p_s$ where $i\le s<j$ has the property that $\psi(p_s)>\psi(p_t)$ for all $i\le t<j$, $t\neq s$. Similarly, the right child of $p_j$ is $p_r$ where $j<r\le k$ has the property that $\psi(p_r)>\psi(p_t)$ for all $j<t\le k$, $t\neq r$.
\end{lem}

\begin{proof}
Same as the proof of Proposition \ref{second leaf corollary}.
\end{proof}

\begin{prop}\label{cor:max min3}
Suppose that $p_j$ is an internal maximum or minimum of $\cT$ and $\psi$ is any $n$-periodic morphism for $\cT$. Then all of the following hold.
\begin{enumerate}
\item The vertical line $j\times \RR$ meets $\overline\psi(\cT)$ at only one point $(j,\psi(p_j))$.
\item There are no edges $(p_i,p_k)$ in $\cT$ where $i<j<k$.
\item All edges of $\cT$ have length $<n$.
\item $\cT$ admits an $n$-periodic morphism of slope zero.
\item $\varepsilon$ is surjective.
\end{enumerate}
\end{prop}

\begin{proof}
Statement (1) follows from Lemma \ref{separation lemma} since $p_j$ has no parents. The other statements follow. For example, (4) follows by the classification of periodic trees (Theorem \ref{cor: classification of trees}) and (4) implies (5).
\end{proof}


%
\begin{figure}[htbp]
\begin{center}
%
{
\setlength{\unitlength}{1cm}
{\mbox{
\begin{picture}(12,4.9)
      \thicklines
\put(-3,-1.5){
\qbezier(2.6,1.4)(4.8,1.7)(7,2)
\put(1.7,1.5){$\cdots$}
}
\put(0,0){
\put(-.05,1.2){$-$}\put(3.9,.7){$+$}\put(4.9,.2){$+$}
\qbezier(0,1)(2,.75)(4,.5)
\qbezier(4,.5)(4.5,.25)(5,0)
\qbezier(0,1)(3.5,1.5)(7,2)\put(4.9,-.1){$\bullet\ p_3$}
\put(3.9,.4){$\bullet$}\put(3.85,.1){$p_2$}
\put(-.1,.92){$\bullet$}\put(0,.7){$p_{-2}$}
}
\put(3,1.5){
\put(-.05,1.2){$-$}\put(3.9,.7){$+$}\put(4.9,.2){$+$}
\qbezier(0,1)(2,.75)(4,.5)
\qbezier(4,.5)(4.5,.25)(5,0)
\qbezier(0,1)(3.5,1.5)(7,2)\put(4.9,-.1){$\bullet\ p_6$}
\put(3.9,.4){$\bullet$}\put(3.85,.1){$p_5$}
\put(-.1,.92){$\bullet$}\put(0,.7){$p_1$}
}
\put(6,3){
\put(-.05,1.2){$-$}\put(3.9,.7){$+$}\put(4.9,.2){$+$}
\qbezier(0,1)(2,.75)(4,.5)
\qbezier(4,.5)(4.5,.25)(5,0)
\qbezier(0,1)(3.5,1.5)(6.3,1.9)
\put(4.9,-.1){$\bullet\ p_9$}
\put(3.9,.4){$\bullet$}\put(3.85,.1){$p_8$}
\put(-.1,.92){$\bullet$}\put(0,.7){$p_4$}
\put(-1.5,-1.7){$\ell_3$}
\put(-.4,0.1){$\ell_3^+$}
\put(-.8,-.6){$\ell_1$}
\put(1.4,-1){$\ell_2$}
\put(6.7,1.5){$\cdots$}
}
\end{picture}}
}}
\caption{Example of leaf formula on positive slope tree. This is the unique $3$-periodic tree with $\pi(k)=5,1,0,8,4,3,\cdots$ and $\varepsilon(k)=-,+,+,-,+,+,\cdots$ for $k=1,2,3,\cdots$. By Remark \ref{rem: dual leaf formula}, $p_3$ is a leaf since $\pi(3)<\pi(2),\pi(4),\pi(5)$. There are three edges $\ell_1,\ell_2,\ell_3$ of lengths 4,1,7. $\ell_3^+$ is a translation of $\ell_3$. This is a positive slope tree in Classification Theorem \ref{cor: classification of trees}.}
\label{fig1}
\end{center}
\end{figure}
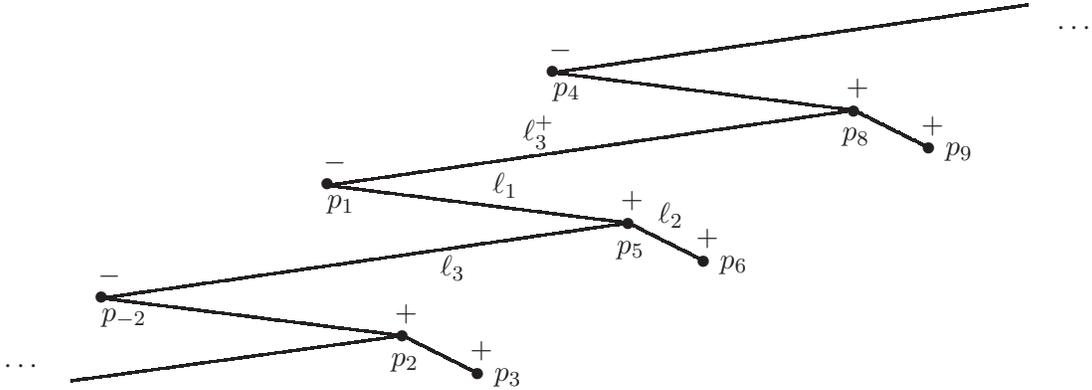
%


\subsection{Periodic trees corresponding to periodic morphisms}\label{ss1.5: periodic trees from periodic morphisms}

The purpose of this subsection is to prove the following theorem which is equivalent to the statement that the regions $R(\cT)$ are disjoint for nonisomorphic $\cT$ and their union is dense in $\RR^{n+1}$.

\begin{thm}\label{thm:every pi gives a unique T}
Given any $\varepsilon$ and any $n$-periodic function $\pi:\ZZ\to \RR$ of slope $\frac mn\neq0$ taking distinct values ($\pi(i)\neq \pi(j)$ for $i\neq j$), there is a unique $n$-periodic tree $\cT$ with sign function $\varepsilon$ on which $\psi(p_i)=\pi(i)$ is a periodic morphism.
\end{thm}

\begin{proof}
The theorem holds for $n=1$ since the partial ordering is a total ordering and $\cT$ is a straight line. So, suppose $n\ge2$. Then we will show by induction on $n$ that there is a unique $n$ periodic tree $\cT$ with sign function $\varepsilon$ and prescribed $n$-periodic morphism $\psi$. The proof breaks up into two cases depending on whether or not $\cT$ has a leaf.
\vs1

\emph{Case 1:} Suppose that $\psi$ satisfies the condition in Proposition~\ref{second leaf corollary}(2) or its dual which is equivalent to $\cT$ having a leaf $p_j$ when $\cT$ exists. \vs1

(Uniqueness) By symmetry, we may assume $\varepsilon_j=-$. Then $p_j$ has no parent in any tree $\cT$ and the unique child $p_s$ of $p_j$ is determined by the function $\psi$. If we remove this leaf and the abutting edge (and all of their translates and renumber the nodes to fill in the gaps), we will obtain an $n-1$ periodic tree $\cT'$ which is unique by induction on $n$. Since $\cT$ is obtained from $\cT'$ by adding a leaf $p_j$ attached to the point $p_s$, it is also unique.

(Existence). Remove the coset $j+n\ZZ$ from the integers and renumber $\ZZ$ to fill in the gap. Then $\pi$ induces an $n-1$ periodic function $\pi'$ on the new set and therefore there exists an $n-1$ periodic tree $\cT'$ with this periodic morphism. The tree $\cT$ is obtained by adding the node $p_j$ (and translates) to $\cT'$ and edge from $p_j$ to $p_s$. So, $\cT$ exists in Case 1.\vs1

\emph{Case 2:} Suppose we are not in Case 1, i.e., $\pi$ is such that $\cT$ would have no leaves if it were to exist.\vs1

(Uniqueness) If $\cT$ exists it is homeomorphic to a line since any branches would terminate in a leaf. So, any local maxima or minima in $\cT$ are internal. Let $M$ be the set of $j\in\ZZ$ at which, by Lemma \ref{cor:max min2}, $p_j$ is an internal max or min of any tree $\cT$. ($M$ might be empty.) Then, given any two consecutive numbers $i,j\in M$ and any integers $a,b\in [i,j]$, we must have $p_a<p_b$ if and only if $\pi(a)<\pi(b)$ (since there are no internal maxima or minima in the open interval $(a,b)$). Furthermore, vertices not in the same such interval $[i,j]$ cannot be related in the partial ordering by Proposition \ref{cor:max min3}. Therefore $\cT$ is unique if it exists.

(Existence) To prove existence, it suffices to show that the partial ordering on the set $\{p_i\}$ described in the previous paragraph satisfies the definition of an $n$-periodic tree. So, let $i,j$ be consecutive point in the set $M$ and suppose by symmetry that $\varepsilon_i=-$. (When $M$ is empty, we will assume by symmetry that $m>0$. So, $\pi$ has positive slope.) Let $S_-=\{s_1<s_2<\cdots\}$ be the set of all $i\le s_k\le j$ so that $\varepsilon_{s_k}=-$ starting with $s_1=i$. When $M$ is empty, $S_-$ also contains $s_p$ for $p<0$.
\vs1

\emph{Claim 1:} The value of $\pi$ is monotonically increasing on the set $S_-$, i.e., $\pi(s_k)<\pi(s_{k+1})$.\vs1

Pf: By construction of $S_-$, the open interval $(s_k,s_{k+1})$ contains no elements of $M$. Therefore, by the characterization of elements of $M$ (Lemma \ref{cor:max min2}(2)), we have 
\begin{equation}\label{eq: part of Claim 1}
\pi(t)<\max(\pi(s_k),\pi(s_{k+1}))\text{ for all }s_k<t<s_{k+1}.
\end{equation}
However, when $k=1$, we also have $\pi(s_1)<\pi(s_1+1)$ since $s_1=i$ is an internal minimum. (When $M$ is empty, we use the assumption that $m>0$. This implies that $\pi(s_p)<\pi(s_{p+1})$ for some $p$. We renumber $S_-$ so that $p=1$.) This implies $\pi(s_1)<\pi(s_2)$. If the Claim is false, take the smallest $k>0$ so that $\pi(s_k)>\pi(s_{k+1})$. Then $p_{s_k}$ would be a leaf by Proposition~\ref{second leaf corollary}. This is the situation being excluded in Case 2. So, Claim 1 holds.\vs1

Claim 1 implies that $j\notin S_-$. So, $j$ must be an internal maximum with $\varepsilon_j=+$. By the dual of Claim 1 we also have that $\pi$ is monotonically increasing on the set $S_+$ of all $k\in[i,j]$ with $\varepsilon_k=+1$.\vs1

\emph{Claim 2:} $\pi(i)<\pi(k)<\pi(j)$ for all $i< k<j$.\vs1

Pf: We are assuming that $i$ satisfies the dual of Lemma \ref{cor:max min2}(2) which says that $\pi(i)<\pi(k)$ where $k$ is the first element of $S_+$. Since $\pi$ is monotonically increasing on both $S_+$ and $S_-$, $\pi(i)<\pi(k)$ for all $k\in S_-\cup S_+=[i,j]$, $k\neq i$. Similarly, $\pi(k)<\pi(j)$ if $i\le k<j$.\vs1

Claim 2 implies that the Hasse diagram of the partial ordering under consideration is a line (which is a tree). Over the interval $[i,j]$, this line goes from $p_i$ up to $p_j$. We see that T1, T2, T3 in Definition \ref{def: conditions T1234} are satisfied. At each point $s\in S_-$, we have shown, in Claim 1 and \eqref{eq: part of Claim 1} that $\pi(k)<\pi(s)$ for all $i\le k<s$. Therefore, the contrapositive of T4(a) is satisfied. Similarly, T4(b) is also satisfied. Therefore, the partial ordering on $P=\{p_i\}$ defined in the Uniqueness paragraph gives an $n$-periodic tree. So, $\cT$ exists in Case 2.

We have shown existence and uniqueness of $\cT$ in both cases.
\end{proof}

\begin{cor}
For a fixed $n$-periodic sign function $\varepsilon$, the regions $\cR(\cT)$ in $V_n(\ZZ)\cong\RR^{n+1}$ are disjoint and their union is dense.
\end{cor}

\begin{proof} 
Any nonempty open subset of $\RR^{n+1}$ contains a vector whose coordinates are linearly independent over $\QQ$. Such a vector represents an $n$-periodic function $\pi:\ZZ\to \RR$ of nonzero slope which is a monomorphism. By the Theorem, $\pi\in\cR(\cT)$ for a unique $n$-periodic tree $\cT$. Therefore, the sets $\cR(\cT)$ are disjoint and their union is dense in $\RR^{n+1}$.
\end{proof}

By an \emph{$n$-periodic permutation} of $\ZZ$ we mean any bijection $\pi:\ZZ\to\ZZ$ with the property that $\pi(i+n)=\pi(i)+n$ and $\sum_{1\le i\le n}(\pi(i)-i)=0$.
For any $n$-periodic permutation $\pi$ of $\ZZ$, let $\cR(\pi)\subset V_n(\ZZ)\cong \RR^{n+1}$ denote the set of all injective $n$-periodic functions $\psi:\ZZ\to\RR$ so that $\psi(i)<\psi(j)$ if and only if $\pi(i)<\pi(j)$ and let $\cR(-\pi)$ be the set of all injective $n$-periodic functions $\psi:\ZZ\to\RR$ so that $\psi(i)>\psi(j)$ when $\pi(i)<\pi(j)$. It is clear that the regions $\cR(\pi)$ are disjoint and the closure of their union is upper half space (given by $m\ge0$) and the regions $\cR(-\pi)$ are also disjoint from each other and from any $\cR(\pi')$ and the closure of their union is the set of all $\psi$ with $m\le0$. 

Classically \cite{Hum:ReflGps}, the sets $\cR(\pi)$, intersected with an $n-1$ dimensional affine plane, are studied as the open cells of a simplicial decomposition of the affine $n-1$ plane.

\begin{cor}
For each fixed $\cT$, $\cR(\cT)$ contains the regions $\cR(\pi)$ for all $\pi$ so that $\cT(\pi)=\cT$ and $\cR(\cT)\supseteq \cR(-\pi')$ for all $\pi'$ so that $\cT(-\pi')=\cT$. Furthermore, the union of these regions $\cR(\pi)$ and $\cR(-\pi')$ is dense in $\cR(\cT)$. \qed
\end{cor}

\subsection{Mutation of periodic trees}

In this section we will define the mutation $\mu_k\cT$ of an $n$-periodic tree in the direction of the edge $\ell_k$.

We define $\mu_k\cT$ as a directed graph, then show that it satisfies the definition of an $n$-periodic tree with the given sign function. For simplicity of terminology we assume that the edge $\ell_k$ has positive slope. The definition is worded so that the negative slope case is given by switching the words ``left'' and ``right''.

\begin{defn}\label{def: mutation of T}
If $\cT$ is an $n$-periodic tree with edges $\ell_i$ and $\ell_k=(p_a,p_b)$ with $p_a<p_b$ and $p_a$ is to the left of $p_b$. Then let $\mu_k\cT$ be the oriented graph with the same vertex set $P=\{p_i\}$ as $\cT$, with automorphism given by translation by $n$: $p_i\mapsto p_{i+n}$ and with $n$ oriented edges $\ell_i'$ given, up to translation, as follows.
\begin{enumerate}
\item $\ell_k'=(p_b,p_a)$ which is $\ell_k$ oriented in the opposite direction $p_b\to p_a$ as in $\cT$. 
\item If $\ell_i=(p_b,p_c)$ where $p_c$ is the unique/left parent of $p_b$ in $\cT$ (for $\varepsilon_b=+,-$, resp.) and $c\not\equiv a$ mod $n$ then $\ell_i'=(p_a,p_c)$ in $\cT'$.
\item If $\ell_i=(p_d,p_a)$ where $p_d$ is the right/unique child of $p_a$ in $\cT$ (for $\varepsilon_a=+,-$, resp.) and $d\not\equiv b$ mod $n$ then $\ell_i'=(p_d,p_b)$ in $\cT'$.
\item If $\ell_i=(p_b,p_{a+sn})$ where $p_{a+sn}$ is the unique/left parent of $p_b$ in $\cT$ (for $\varepsilon_b=+,-$, resp.) then $\ell_i'=(p_a,p_{b+sn})$ in $\cT'$.
\item $\ell_i'=\ell_i$ if none of the above apply.
\end{enumerate}
\end{defn}

When (4) applies, two edges change. Otherwise, at most three edges change according to (1), (2), (3).

\begin{lem}\label{lem: for tree mutation}
For any edge $\ell_i'=(p_s,p_t)$ in $\cT'$ with $i\neq k$, if $p_s<p_t$ in $\cT'$ then $p_s<p_t$ in $\cT$ and the unique monotonically increasing path from $p_s$ to $p_t$ in $\cT$ consists of $\ell_i$ and 0,1 or 2 translates of the edge $\ell_k$.
\end{lem}

\begin{prop}\label{prop: mutation of tree is tree}
$\mu_k\cT$ is an $n$-periodic tree with sign function $\varepsilon$.
\end{prop}

\begin{proof}
We observe first that $\cT$, $\cT'$ become isomorphic as directed graphs if the edges $\ell_k,\ell_k'$ are collapsed and the vertices $p_a,p_b$ are identified. Let $\overline\cT$ denote this collapsed tree. Since $\cT$ is a tree, so are $\overline \cT$ and $\cT'$. Therefore, $\cT'$ is the Hasse diagram of some $n$-periodic partial ordering on $P=\{p_i\}$. We observe that the corresponding edges $\ell_i,\ell_i'$ in $\cT,\cT'$ has the same image in $\overline\cT$ which we denote $\overline\ell_i$.

It is easy to see that Conditions T1,T2,T3 in Definition \ref{def: conditions T1234} hold for $\cT'$. For example, take T2. When the orientation of $\ell_k=(p_a,p_b)$ is changed then $p_b$ gains a new parent $p_a$ which is a unique or left parent of $p_b$ in $\cT'$ depending of $\varepsilon_b$. However, by (2), any already existing unique or left parent of $p_b$ in $\cT$ becomes a parent of $p_a$ in $\cT'$. Therefore, $p_b$ has the correct number of parents in $\cT'$. Also, $p_b$ loses $p_a$ as left child and possibly gains a child $p_d$. We need to check that $p_d$ becomes a left child of $p_b$ even though $p_d$ was a right child of $p_a$ in case $\varepsilon_b=+$. But, in that case, the edge $(p_d,p_a)$ in $\cT$ cannot cross under the point $p_b$. So, $d<b$ as required.

Similar, $p_a$ gains $p_b$ as a new right or unique child depending on $\varepsilon_a$. But it also loses the right/unique child it already had by (3).

In Case (4) the two movements of parents and children happen simultaneously: the bottom endpoint of $\ell_i$ slides from $p_b$ to $p_a$ and the top endpoint slides from $p_{a+sn}$ to $p_{b+sn}$. Considering these as two separate moves as in (2), (3), we see that each vertex ends up with the correct number of parents and children.

It remains to show that $\cT'$ satisfies T4. Suppose that $p_i<p_j$ in the tree $\cT'$. Then there is a unique path $\gamma$ from $p_i$ to $p_j$ in $\cT'$ which is monotonically increasing. The image $\overline\gamma$ of this path in the collapsed tree $\overline\cT$ will also be monotonically increasing. So, $p_i\le p_j$ in $\overline\cT$. But, the lifting $\tilde\gamma$ of $\overline\gamma$ to a path from $p_i$ to $p_j$ in $\cT$ might be decreasing on some occurrences of the edge $\ell_k=(p_a,p_b)$ and increasing on all other parts of the path. Since $\cT$ is an $n$-periodic tree, the occurrence of $\ell_k$ in the path $\tilde\gamma$ gives it local maxima and minima and forces the beginning and end of the path to be separated by a vertical wall by Lemma \ref{separation lemma}. In other words, either $p_i<p_j$ in $\cT$ or there is a vertical wall separating $p_i,p_j$. This implies that $\cT'$ satisfies T4: Suppose $i<s<j$, $\cT'$ has an edge $\ell'=(p_i,p_j)$ and $p_s<\min(p_i,p_j)$ in $\cT'$. Then we let $z$ be the unique point on $\ell'$ with $x$-coordinate $s$. By Lemma \ref{lem: for tree mutation} there is a monotonically increasing path in $\cT$ from $p_i$ to $p_j$. Let $z'$ be any point on this path with $x$-coordinate $s$. Then $z,z'$ map to two points on the edge $\overline \ell$ in $\overline\cT$. Since $p_s$ is less than all points of $\overline \ell$ in $\overline \cT$ by assumption, $p_s<z$ in $\cT$ since $p_s,z$ cannot be separated by a vertical wall. Therefore $\varepsilon_s=+$ as required by T4(a). T4(b) is similar.
\end{proof}

%
\begin{figure}[htbp]
\begin{center}
%
{
\setlength{\unitlength}{1cm}
{\mbox{
\begin{picture}(12,3.5)
      \thicklines
\put(3,0){
\put(-3.05,1.2){$-$}
\put(3.9,.7){$+$}\put(4.9,.2){$+$}
\qbezier(-3,1)(.5,.75)(4,.5)
\qbezier(4,.5)(4.5,.25)(5,0)
\qbezier(-3,1)(2,1.5)(7,2)
\put(4.9,-.1){$\bullet\ p_6$}
\put(3.9,.4){$\bullet$}\put(3.85,.1){$p_5$}
\put(-3.1,.92){$\bullet$}
\put(-3,.7){$p_{-2}$}
}
\put(6,1.5){
\put(-3.05,1.2){$-$}
\put(3.9,.7){$+$}\put(4.9,.2){$+$}
\qbezier(-3,1)(.5,.75)(4,.5)
\qbezier(4,.5)(4.5,.25)(5,0)
\qbezier(-3,1)(2,1.5)(7,2)
\put(4.9,-.1){$\bullet\ p_9$}
\put(3.9,.4){$\bullet$}\put(3.85,.1){$p_8$}
\put(-3,0){\put(-.1,.92){$\bullet$}\put(0,.7){$p_1$}}
\put(1,.9){$\ell_3'$}
\put(-2,0.15){$\ell_1'$}
\put(1.4,-1){$\ell_2'$}
}
\end{picture}}
}}
\caption{Example of tree mutation, case (4): $\cT'=\mu_3\cT$ where $\cT$ is given in Figure \ref{fig1} with $n=3,a=-2,b=5,s=1$. In $\cT$, $p_{a+sn}=p_1$ is the unique parent of $p_b=p_5$ and, in $\mu_3\cT$, $p_{b+sn}=p_8$ is the right parent of $p_a=p_{-2}$.}
\label{fig1b}
\end{center}
\end{figure}
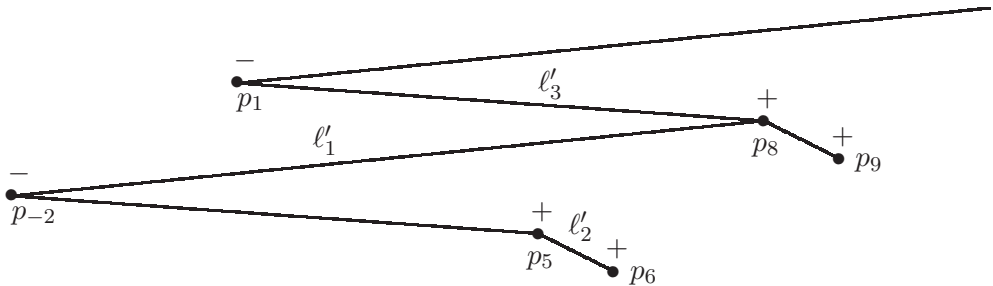
%

%



\section{Edge vectors, semi-invariants and cluster tilting objects}\label{sec2}

We assume from now on that $\varepsilon$ is surjective. (So far we have used this assumption only in Proposition \ref{cor:edge vectors are Schur roots}.)
We define edge vectors of periodic trees and verify the stability conditions of \cite{IOTW}. As a consequence we obtain a bijection between $n$-periodic trees and cluster tilting objects of type $\widetilde{A}_{n-1}$. We will show in the next section that the $c$-vectors corresponding to a cluster tilting object are the \emph{negatives} of the edge vectors of the corresponding periodic tree. 


\subsection{Edge vectors}\label{ss2.1: edge vectors}

We define edge vectors of a periodic tree and derive basic properties which characterize these vectors.

Let $\cT$ be an $n$-periodic tree with vertices $p_i$ and edges $\ell=(p_i,p_j)$ with sign $\delta_\ell$ equal to the sign of the slope of the edge. Then we define the \emph{edge vectors} of $\cT$ to be $\gamma(\ell)=\delta_\ell \beta(\ell)$ where
\[
	\beta(\ell)=\beta_{ij}:=e_{\overline{i+1}}+e_{\overline{i+2}}+\cdots+e_{\overline{j}}\in \ZZ^n
\]
where $\overline i$ is one plus the reduction of $i-1$ modulo $n$ and $e_i$ is the $i$th unit vector of $\ZZ^n$. Note that $\beta(\ell)=|\gamma(\ell)|$ and that the edge vectors determine $\cT$ since they give all the edges of $\cT$ and their orientation.

\begin{eg}
In Figure \ref{fig1} the edges are (up to translation by $n=3$) $(p_1,p_5)$ and $(p_2,p_3)$ with negative slope and $(p_1,p_8)$ with positive slope. Therefore the edge vectors are the column vectors $\gamma(\ell_1)=-\beta_{15}=(-1,-2,-1)^t$, $\gamma(\ell_2)=-\beta_{23}=(0,0,-1)^t$ and $\gamma(\ell_3)=\beta_{18}=(2,3,2)^t$.
\end{eg}

\begin{defn}\label{def: FRn+1 to Rn}
Let $F:V_n(\ZZ)\cong\RR^{n+1}\to\RR^n$ be the linear map which sends an $n$-periodic function $\pi:\ZZ\to\RR$ to the vector $y\in\RR^n$ with coordinates $y_i=\pi(i)-\pi(i-1)$ for $1\le i\le n$. Note that the sum of the $n$ coordinates of $y$ is $\sum y_i=m=\pi(n)-\pi(0)$.
\end{defn}

\begin{prop}
For any $n$-periodic tree $\cT$, the subset $\cR(\cT)\subseteq V_n(\ZZ)\cong\RR^{n+1}$ is the inverse image under $F$ of the set of all $y\in\RR^n$ satisfying $y^t \gamma(\ell)>0$ for all edge vectors $\gamma(\ell)$ of $\cT$.
\end{prop}

\begin{proof}
By definition, $\cR(\cT)$ is the set of all $n$-periodic functions $\psi:P_\cT\to \RR$ with the property that $\psi(p_i)<\psi(p_j)$ for any edge $\ell=(p_i,p_j)$ with positive slope and $\psi(p_i)>\psi(p_j)$ if $\ell$ has negative slope. One formula can be used for both cases:
\[
	\delta_\ell(\psi(p_j)-\psi(p_i))>0\,.
\]
Using the vector $y=F(\psi)$, this can be written as:
\[
	\delta_\ell(\psi(p_j)-\psi(p_i))=\delta_\ell(y_{\overline{i+1}}+\cdots+y_{\overline{j}})=\delta_\ell y^t\beta_{ij}=y^t \gamma(\ell)>0\,
\]
where $\overline k\equiv k$ mod $n$, $1\le \overline k\le n$. The proposition follows.
\end{proof}

Note that the linear condition $y^t\gamma(\ell)=0$ on $y=F(\psi)$ is equivalent to the condition $\psi(p_i)=\psi(p_j)$.

\begin{prop}\label{sign coherence of c-inverse}
An $n$-periodic tree $\cT$ has exactly $n$ distinct edge vectors $\gamma_1,\cdots,\gamma_n\in \ZZ^n$ with the following properties.
\begin{enumerate}
\item The sum of the coordinates of $\gamma_i$ is not divisible by $n$ for any $i$.
\item The determinant of the $n\times n$ integer matrix $\Gamma_\cT=[\gamma_1,\cdots,\gamma_n]$ is $\pm1$.
\item All nonzero entries in any column of $\Gamma_\cT^{-1}$ have the same sign.
\end{enumerate}
\end{prop}

\begin{rem}
We call $\Gamma_\cT$ the \emph{edge matrix} of $\cT$. It is well defined up to permutation of the columns.
\end{rem}

\begin{proof}
(1) follows from Proposition \ref{cor:edge vectors are Schur roots}.

To prove (2) we will show that each unit vector $e_i\in\ZZ^n$ is an integer linear combination of edge vectors. To do this, take the path $\lambda$ from $p_{i-1}$ to $p_{i}$ in the tree $\cT$. This path is a sequence of edges $\ell_1,\ell_2,\cdots,\ell_m$ in $\cT$. This gives an equation of the form:
\[
	e_i=\sum\pm \gamma(\ell_j)
\]
where each sign $\pm$ tells whether $\lambda$ goes up or down along the edge $\ell_j$ (since $\pm\gamma(p_i,p_j)=\beta_{kj}-\beta_{ki}$ for $k<i,j$ making $\sum \pm\gamma(\ell_j)=\beta_{ki}-\beta_{k,i-1}=e_i$). This proves (2).

Statement (3) is that all terms in this sum have the same sign. Equivalently, the path $\lambda$ is either monotonically increasing or monotonically decreasing. This is true when $\lambda$ is the single edge $(p_{i-1},p_i)$. In other cases, $\lambda$ must contain an edge $\ell$ which straddles either $p_i$ or $p_{i-1}$, say the latter. If $\ell$ has positive slope then, by Corollary \ref{monotonic cor}, the part of $\lambda$ which goes from $p_{i-1}$ to $\ell$ is monotonically increasing. For the same reason, the remainder of the path $\lambda$ is also monotonically increasing. Similarly if $\ell$ has negative slope. This proves (3).
\end{proof}

\begin{eg}
In Figure \ref{fig1}, the edge matrix and its inverse are given by:
\[
	\Gamma_\cT=\mat{-1&0&2\\
	-2& 0&3\\
	-1&-1&2
	}\,,\quad \Gamma_\cT^{-1}=\mat{3&-2&0\\
	1&0&-1\\
	2&-1&0}
\]
The columns of $\Gamma_\cT^{-1}$ indicate the paths from $p_{i-1}$ to $p_{i}$. For example, the path from $p_3$ to $p_4$ goes up the edge $\ell_2=(p_3,p_2)$, up the edge $\ell_1=(p_5,p_1)$ and two of its translates, and it goes up along two translates of the edge $\ell_3=(p_1,p_8)$, giving the first column of $\Gamma_\cT^{-1}$ as $(3,1,2)^t$.
\end{eg}

\begin{cor}\label{FR(T) is an octant} For any periodic tree $\cT$, the linear isomorphism $\RR^n\to \RR^n$ whose matrix is the transpose $\Gamma_\cT^t$ of the edge matrix of $\cT$ sends the region $F\cR(\cT)$, resp. $F\overline \cR(\cT)$, to the set of all vectors in $\RR^n$ whose coordinates are all positive, resp. nonnegative.
\end{cor}

\begin{proof}
The conditions $y^t\gamma_i>0$ which characterize $F\cR(\cT)$ are equivalent to $\gamma_i^t y>0$.
\end{proof}


\subsection{Representations}\label{ss2.2: representations}
We describe the edge vectors of a periodic tree in terms of representations of a quiver given by the sign function of the tree.

Given a (surjective) $n$-periodic sign function $\varepsilon$, we have a quiver $\widetilde{A}_{n-1}^\varepsilon$ with $n$ vertices and $n$ edges forming one cycle as follows. The vertices of $\widetilde{A}_{n-1}^\varepsilon$ are $1,2,\cdots,n$. For each $i$ there is one arrow between ${i}$ and $\overline{i+1}$ which goes to the left if $\varepsilon_i$ is positive and goes to the right if $\varepsilon_i$ is negative. For example, for $\varepsilon=(-,+,+)$ we have
\[
\xymatrixrowsep{1pt}
\xymatrix{
\\
Q_{-++}:&1\ar@/^1pc/[rr] \ar[r] & 
	2 &
	3\ar[l]\,.	}
\]
By assumption, the signs $\varepsilon_i$ are not all equal. So $\widetilde{A}_{n-1}^\varepsilon$ has no oriented cycles and $\kk\widetilde{A}_{n-1}^\varepsilon$ is a finite dimensional hereditary algebra over any field $\kk$. 

We recall that a \emph{representation} $M$ of $\widetilde{A}_{n-1}^\varepsilon$ consists of a vector space $M_i$ at each vertex, which we always assume to be finite dimensional, and a linear map $M_i\to M_j$ for every arrow $i\to j$ in the quiver. Representations are equivalent to modules over the ring $\kk\widetilde{A}_{n-1}^\varepsilon$. The \emph{dimension vector} of a representation $M$ is
$
\underline\dim M:=(\dim_\kk M_1,\dim_\kk M_2,\cdots,\dim_\kk M_n)\in\NN^n$.

Isomorphism classes of representations $M$ of $\widetilde{A}_{n-1}^\varepsilon$ are in bijection with homotopy classes of monomorphisms $p:P_1\to P_0$ between projective $\kk\widetilde{A}_{n-1}^\varepsilon$-modules. The correspondence sends $M$ to it projective representation and $f$ to its cokernel. A \emph{virtual representation} of $\widetilde{A}_{n-1}^\varepsilon$ is defined to be the homotopy class of a not necessarily injective morphism $f:P_1\to P_0$ between projective modules $P_i$. As an example, take the morphism $P\to 0$ for any projective $P$. We denote this virtual representation as $P[1]$ and call it a \emph{shifted projective}. It is easy to see that any indecomposable virtual representation of $\widetilde{A}_{n-1}^\varepsilon$ is either a standard representation or a shifted projective. The dimension vector of a virtual representation $P_1\to P_0$ is defined to be $\underline\dim P_0-\underline\dim P_1\in\ZZ^n$.

The \emph{Euler matrix} $E_\varepsilon$ is the matrix with 1's on the diagonal, $ij$ entry equal to $-1$ if there is an arrow $i\to j$ in $\widetilde{A}_{n-1}^\varepsilon$ and 0 elsewhere. For example,
\[
E_{-++}=\mat{1 & -1 & -1\\
0&1&0\\
0&-1&1}\,.
\]
The Euler matrix of a quiver without oriented cycles is invertible and therefore gives a nondegernate form $\left<\cdot,\cdot\right>:\ZZ^n\times \ZZ^n\to\ZZ$ called the \emph{Euler-Ringel form} given by
\[
	\left<x,y\right>:=x^tE_\varepsilon y\,.
\]
This form has the property that, for any two representations $M,N$ of $\widetilde{A}_{n-1}^\varepsilon$,
\begin{equation}\label{Ringel form}
	\left<\underline\dim M,\underline\dim N\right>=\dim\Hom(M,N)-\dim\Ext(M,N)\,.
\end{equation}

The following calculation shows that the columns $\pi_j$ of the matrix $(E_\varepsilon^t)^{-1}$ are the dimension vectors of the indecomposable projective $\kk\widetilde{A}_{n-1}^\varepsilon$-modules. 
\[
	\left<\pi_j,\underline\dim M\right>=\pi_j^t E_\varepsilon\underline\dim M=e_j^t\underline\dim M=\dim M_j=\dim\Hom(P_j,M)-\dim\Ext(P_j,M).
\]
In the example, these are $\pi_1=(1,2,1)^t,\pi_2=(0,1,0)^t,\pi_3=(0,1,1)^t$. These vectors are called the \emph{projective roots} of $\widetilde{A}_{n-1}^\varepsilon$.

The \emph{positive roots} of $\widetilde{A}_{n-1}^\varepsilon$ are those of the form $\beta_{ij}$ where $i<j$ and $j-i$ is not divisible by $n$. For example, the projective roots are always positive roots. The \emph{negative roots} are $-\beta_{ij}$ where $\beta_{ij}$ is a positive root. By Proposition \ref{cor:edge vectors are Schur roots}, the edge vectors of a periodic tree are positive and negative roots. However, only certain ones called ``real Schur roots'' occur. These are defined as follows

\begin{defn}\cite{K},\cite{S92},\cite{DW}.
A (positive) \emph{Schur root} of $\widetilde{A}_{n-1}^\varepsilon$ is a vector $\beta\in\NN^n$ so that the general representation of $\kk\widetilde{A}_{n-1}^\varepsilon$ with dimension vector $\beta$ has endomorphism ring $\kk$. In particular, the representation is indecomposable. The Schur root $\beta$ is called \emph{real} or \emph{imaginary} depending on whether $\left<\beta,\beta\right>>0$ or $\le 0$ respectively. In the first case, the general representation of dimension vector $\beta$ is \emph{rigid}, i.e., has no self-extensions. In the second case, it has self-extensions. A rigid indecomposable module is called \emph{exceptional}.
\end{defn}

\begin{thm}
The Schur roots of $\widetilde{A}_{n-1}^\varepsilon$ are given as follows.
\begin{enumerate}
\item[(0)] \emph{(null root)} The null root $\eta=\beta_{0n}$.
\item \emph{(preprojective roots)} $\beta_{ij}$ where $(\varepsilon_i,\varepsilon_j)=(-,+)$ and $i<j$.
\item \emph{(preinjective roots)} $\beta_{ij}$ where $(\varepsilon_i,\varepsilon_j)=(+,-)$ and $i<j$.
\item \emph{(regular roots)} $\beta_{ij}$ where $\varepsilon_i=\varepsilon_j$ and $i<j<i+n$.
\end{enumerate}
The null root is an imaginary root and the others are the real Schur roots.
\end{thm}

We will show that every real Schur root occurs as a edge vector of some periodic tree. Multiples of the null root are called roots, but they are not Schur roots.

\begin{prop}
Given a real Schur root $\beta_{ij}$ there is an exceptional representation $M_{ij}$ with dimension vector $\beta_{ij}$ which is unique up to isomorphism. Let $P_1\to P_0\to M_{ij}$ be a minimal projective presentation of $M_{ij}$. Then the number of summands of $P_0$ minus the number of summands of $P_1$ is positive if $\beta_{ij}$ is preprojective, negative if $\beta_{ij}$ is preinjective and zero if $\beta_{ij}$ is regular.
\end{prop}

More generally, for any $a<b$, let $M_{ab}$ denote the string module with dimension vector $\beta_{ab}$ which comes from a generic indecomposable finite dimensional representation of the infinite covering quiver $\widetilde{A_\varepsilon}$ of $\widetilde{A}_{n-1}^\varepsilon$. We say that $\beta_{ab}$ is a \emph{subroot} of $\beta_{ij}$ and we write $\beta_{ab}\subseteq \beta_{ij}$ if $M_{ab}$ is isomorphic to a submodule of $M_{ij}$.

\begin{lem}\label{subroot lemma}
Suppose that $i\le a<b\le j$. Then $\beta_{ab}\subseteq\beta_{ij}$ if and only if the following are satisfied for the same integer $s$.
\begin{enumerate}
\item Either $a=i+sn$ or $\varepsilon_{a}=-1$.
\item $b=j+sn$ or $\varepsilon_b=+1$.
\end{enumerate}
\end{lem}

\begin{proof}
These are the conditions which make the arrows in the quiver point inward towards the support of $M_{ab}\cong M_{a-sn,b-sn}$ making it a submodule of $M_{ij}$.
\end{proof}

\begin{lem}\label{lem: second stability condition holds}
Suppose that $\cT$ is an $n$-periodic tree and $\ell=(p_i,p_j)$ is an edge in $\cT$. Let $\beta_{ab}$ be any subroot of $\beta_{ij}=\beta(\ell)$. Then
\[
	F(\psi)^t\beta_{ab}<0
\]
for any $\psi\in \cR(\cT)$.
\end{lem}

\begin{rem}\label{second stability condition is satisfied on FR(T)}
When we go to the closure of $\cR(\cT)$, the inequality could become an equality. So, we conclude that $y^t\beta_{ab}\le0$ for all $y\in F\overline \cR(\cT)$.
\end{rem}

\begin{proof}
Consider the case $i<a<b<j$, the other cases being similar. Then $\varepsilon_a=-1$ and $\varepsilon_b=+1$. By Corollary \ref{monotonic cor} we have
\[
	\psi(p_a)>\max(\psi(p_i),\psi(p_j))>\min(\psi(p_i),\psi(p_j))>\psi(p_b)
\]
which implies that that the point $(b,\psi(p_b))\in \RR^2$ is below and to the right of $(a,\psi(p_a))$. This is equivalent to the equation $F(\psi)^t\beta_{ab}<0$.
\end{proof}


\subsection{Semi-invariants and cluster tilting objects}\label{ss2.3: semi-invariants and cluster tilting objects} This subsection contains the main result of Section \ref{sec2}: the 1-1 correspondence between $n$-periodic trees and cluster tilting objects in the cluster category of $\kk \widetilde{A}_{n-1}^\varepsilon$. The correspondence, given in Theorems \ref{virtual stability theorem} and \ref{thm:correspondence between trees and cluster tilting objects} also gives, in Corollary \ref{cor: components of cluster from topology of tree}, a description of which components of the cluster tilting object are regular, preprojective and preinjective or shifted projective depending on the geometry of the periodic tree.

The Stability Theorem for virtual semi-invariants from \cite{IOTW} characterizes vectors in the support of a semi-invariant. There are several equivalent versions of the stability conditions which we now review. To simplify the logical development of this subject we use these equivalent formulas as a definition. The original definition of a virtual semi-invariant is Theorem \ref{original def of virtual semi-invariant} below.

\begin{prop}\label{prop: equivalent versions of stability conditions}
Suppose that $\beta_{ij}$ is a real Schur root and $v\in\RR^n$ so that $\left<v,\beta_{ij}\right>=0$. Then the following are equivalent.
\begin{enumerate}
\item $\left<v,\beta\right>\le 0$ for all proper subroots $\beta\subsetneq\beta_{ij}$.
\item $\left<v,\beta\right>\le 0$ for all proper subroots $\beta\subsetneq\beta_{ij}$ of the form $\beta=\beta_{aj}$ and $\beta=\beta_{ib}$.
\item $\left<v,\beta\right>\le 0$ for all real Schur subroots $\beta\subsetneq\beta_{ij}$.
\end{enumerate}
Furthermore, these conditions are still equivalent if we replace $\le$ with $<$.
\end{prop}

In the following proof and in the rest of this paper we will use the correspondence between vectors $v\in\RR^n$ and $n$-periodic functions $\pi_v:\ZZ\to\RR$ given by $E_\varepsilon^tv=y=F\pi_v$. Then
\[
	\left<v,\beta_{ij}\right>=v^tE_\varepsilon\beta_{ij}=y^t\beta_{ij}=y_{i+1}+\cdots+y_j=\pi_v( j)-\pi_v( i).
\]
So, the condition $\left<v,\beta_{ij}\right>=0$ on $v$ is equivalent to the condition $\pi_v(i)=\pi_v(j)$ and the condition $\left<v,\beta_{ij}\right><0$ is equivalent to the condition $\pi_v(j)<\pi_v(i)$.

\begin{proof} We are given that $\pi_v(i)=\pi_v(j)$. Condition (1) implies (2) and (3) since these are special cases of (1). So it suffices to show $(2) \then (1)$ and $(3) \then (2)$.

By Lemma \ref{subroot lemma}, Condition (2) is equivalent to the following condition on $\pi_v$:\vs2

$(2')$ \quad $\pi_v(a)\ge \pi_v(j)$ for all $i<a<j$ with $\varepsilon_a=-$ and $\pi_v(b)\le \pi_v(i)$ for all $i<b<j$ with $\varepsilon_b=+$. \vs2

\noi Since $\pi_v( i)=\pi_v( j)$, this implies that $\left<v,\beta_{ab}\right>=\pi_v( b)-\pi_v( a)\le0$ for any proper subroot $\beta_{ab}\subsetneq \beta_{ij}$ which is not of the form $\beta_{ia}$ or $\beta_{ib}$ and $\pi_v(b)<\pi_v(a)$ if the inequality in $(2')$ is strict. Therefore, (2) implies (1).

Finally, we will show that (3) implies $(2')$. In the case where $\beta_{ij}$ is either preprojective or has length $<n$, all subroots of $\beta_{ij}$ will be real Schur roots and (1), (3) are equivalent. Therefore, we may assume that $\beta_{ij}$ is preinjective and $i+n<j$. Let $k\ge1$ be maximal so that $i+kn<j$. Then 
\[
i<j-kn<i+n\le i+kn<j
\]
and $\beta_{j-kn,i+kn}$ is a preprojective subroot of $\beta_{ij}$. So, (3) implies
\[
\left<v,\beta_{j-kn,i+kn}\right>=\pi_v({i+kn})-\pi_v({j-kn})\le0\,.
\]
This implies $m\le0$ where $m=\left<v,\eta\right>=y_1+\cdots+y_n=\pi_v({i+n})-\pi_v({i})$ since $m>0$ would give $\pi_v({i+kn})>\pi_v( i)=\pi_v( j)>\pi_v({j-kn})$, a contradiction. 

To prove $(2')$, let $i<a<j$ with $\varepsilon_a=-$. Let $s\ge0$ be maximal so that $a+sn\le j$. Then
\[
	\pi_v( a)\ge\pi_v({a+sn})\ge\pi_v( j)
\]
since $\beta_{j,a+sn}$ is a real Schur subroot of $\beta_{ij}$. Similarly, $\pi_v( b)\le\psip i$ if $i<b<j$ with $\varepsilon_b=+$. So (3) implies $(2')$ and if the inequality in (3) is strict then the inequality in $(2')$ is strict.
\end{proof}

\begin{defn}\label{def: real support D(b)}
Suppose that $\beta_{ij}$ is a real Schur root of $\widetilde{A}_{n-1}^\varepsilon$. Then the \emph{real support} $D(\beta_{ij})$ of the associated semi-invariant is defined to be the set of all vectors $v\in \RR^n$ satisfying the following.
\begin{enumerate}
\item $\left<v,\beta_{ij}\right>=0$.
\item $\left<v, \beta_{ab} \right>\le0$ for all real Schur subroots $\beta_{ab}\subseteq \beta_{ij}$ (and thus for all subroots by the proposition above).
\end{enumerate}
These are called the \emph{Stability Conditions} (on $v$). Let $H(\beta_{ij})$ denote the hyperplane in $|RR^n$ of all vectors $v$ satisfying (1). By the \emph{interior} of $D(\beta_{ij})$ we mean its interior as a subset of this hyperplane.
\end{defn}

\begin{lem}\label{lem:D(b) has nonempty interior for b real Schur root}
If $\beta_{ij}$ is a real Schur root of $\widetilde{A}_{n-1}^\varepsilon$ then $D(\beta_{ij})$ has a nonempty interior. In other words, there is a $v\in D(\beta_{ij})$ so that $\left<v, \beta \right><0$ for all proper subroots $\beta\subsetneq \beta_{ij}$.
\end{lem}

\begin{proof}
By symmetry we may assume that $\varepsilon_i=-$. So, $\beta_{ij}$ is either preprojective or regular. Let $v\in \RR^n$ be the vector corresponding to the $n$-periodic function $\pi_v$ given by
\[
	\pi_v( k)=\begin{cases}    \left[\frac{k-j}n\right] & \text{if $\varepsilon_k=+$ or if $k\equiv j$ modulo $n$}\\
\left[\frac{k-i+n-1}n\right] & \text{otherwise}
    \end{cases}
\]
where $[\cdot]$ is the greatest integer function. Then $\pi_v( j)=0=\psip i$. So, $\left<v,\beta_{ij}\right>=\pi_v( j)-\pi_v( i)=0$. If $i<k<j$, then the case where $k\equiv j$ mod $n$ and $\varepsilon_k=-$ cannot occur since $\varepsilon_i=-=\varepsilon_j$ only in the regular case $j-n<i$. Therefore, we can delete the ``or if $k\equiv j$ modulo $n$'' clause in the definition of $\pi_v( k)$. Consequently, if $\varepsilon_k=-$, then $\pi_v( k)\ge1$ making $\left<v,\beta_{kj}\right>\le -1$ and, if $\varepsilon_k=+$, $\pi_v( k)\le-1$ making $\left<v,\beta_{ik}\right>\le-1$ in that case. Therefore, $\left<v,\beta\right>\le-1$ for any proper subroot $\beta\subsetneq\beta_{ij}$.
\end{proof}

We observe that, since $\beta_{ij}$ has only a finite number of subroots, $D(\beta_{ij})$ is a closed convex polyhedral region in the hyperplane in $\RR^n$ given by (1). Proper subroots $\beta\subsetneq \beta_{ij}$ are not collinear with $\beta_{ij}$. So, for such $\beta$, we have strict inequalities $\left<v,\beta\right><0$ for all $v$ in the interior of $D(\beta_{ij})$ which is nonempty by the lemma above.

In terms of the corresponding $n$-periodic function $\pi_v$, the set $D(\beta_{ij})$ and its interior can be very usefully described as follows.

\begin{prop}\label{prop: psi satisfies stability conditions for D(b)}
A vector $v\in \RR^n$ lies in $D(\beta_{ij})$ if and only if the corresponding function $\pi_v:\ZZ\to\RR$ satisfies the following.
\begin{enumerate}
\item $\pi_v(i)=\pi_v(j)$
\item $\pi_v(a)\ge\pi_v(j)$ for all $i<a<j$ with $\varepsilon_a=-$
\item $\pi_v(b)\le \pi_v(i)$ for all $i<b<j$ with $\varepsilon_b=+$.
\end{enumerate}
Furthermore, $v$ lies in the interior of $D(\beta_{ij})$ (as a subset of the hyperplane given by (1)) if and only if all of the inequalities in (2) and (3) are strict.\qed
\end{prop}

Let $D_B(\beta_{ij})=D(\beta_{ij})\cap B^n$ for any subring $B$ of $\RR$. Then $D_\QQ(\beta_{ij})$ is dense in $D(\beta_{ij})$ and contains the zero vector. So, $D(\beta_{ij})$ is the closure of the convex hull of $D_\ZZ(\beta_{ij})$ in $\RR^n$. Therefore, the following theorem gives the representation theoretic meaning of $D(\beta_{ij})$.

\begin{thm}\cite{IOTW}\label{original def of virtual semi-invariant}
A vector $v$ lies in $D_\ZZ(\beta_{ij})$ if and only if there exists a virtual representation $V:P_1\to P_0$ of dimension vector $v$ with the property that
\[
	V^\ast:\Hom(P_0,M_{ij})\to \Hom(P_1,M_{ij})
\]
is an isomorphism where $M_{ij}$ is the unique exceptional module with dimension vector $\beta_{ij}$.
\end{thm}

\begin{rem} Since ``$V^\ast$ is an isomorphism'' is a Zariski open condition on $V$ in the affine space $\Hom(P_1,P_0)$, the existence of one such $V$ implies that the general element of $\Hom(P_1,P_0)$ has this property.
Also, this condition is equivalent to the condition that $\Hom_{\cD^b}(V,M_{ij})=0=\Ext_{\cD^b}^1(V,M_{ij})$ in the bounded derived category ${\cD^b}={\cD^b(mod\text-\kk\widetilde A_{n-1}^\varepsilon)}$ of $mod\text-\kk\widetilde A_{n-1}^\varepsilon$ since $\Hom_{\cD^b}(V,M_{ij})=\ker V^\ast$ and $\Ext_{\cD^b}^1(V,M_{ij})=\coker V^\ast$.
\end{rem}

\begin{lem}\label{first stability lemma}
Suppose that $\cT$ is an $n$-periodic tree with sign function $\varepsilon$. Let $\gamma_1,\cdots,\gamma_n$ be the edge vectors of $\cT$. Then the boundary $\del F\cR(\cT)$ of the open region $F\cR(\cT)\subseteq\RR^n$ is the union of $n$ sets $\del_kF\cR(\cT)$ satisfying the following.
\begin{enumerate}
\item $y\in \del_kF\cR(\cT)$ if and only if
	\begin{enumerate}
	\item $y^t \gamma_k=0$ and
	\item $y^t \gamma_i\ge0$ for all $i\neq k$.
	\end{enumerate}
\item $\del_k F\cR(\cT)\subseteq E_\varepsilon^t D(|\gamma_k|)$.
\end{enumerate}
\end{lem}

We call $\del_kF\cR(\cT)$ the \emph{face} of $F\cR(\cT)$ corresponding to the edge $\ell_k$.

\begin{proof}
(1) follows immediately from Corollary \ref{FR(T) is an octant}. To prove (2), we need to show that $(E_\varepsilon^t)^{-1}y$ satisfies the stability conditions defining $D(|\gamma_k|)$. The first stability condition holds since $y^t \gamma_k=\left<(E_\varepsilon^t)^{-1}y,\gamma_k\right>=0$. The second stability condition holds by Remark \ref{second stability condition is satisfied on FR(T)}.
\end{proof}

\begin{lem}\label{second stability lemma}
For any real Schur root $\beta_{ij}$ of $\widetilde{A}_{n-1}^\varepsilon$, the set $E_\varepsilon^t D(\beta_{ij})$ is equal to the closure of the union of the sets $\del_k F\cR(\cT)$ for all $n$-periodic trees $\cT$ for which $|\gamma_k|=\beta_{ij}$ for some numbering of the edge vectors of $\cT$. In particular, the open sets $F\cR(\cT)$ are disjoint from all sets of the form $E_\varepsilon^tD(\beta_{ij})$.
\end{lem}

\begin{proof}
For any $y\in E_\varepsilon^t D(\beta_{ij})$ and any open neighborhood $U$ of $y$ in $\RR^n$, we will find another point $y'\in U\cap E_\varepsilon^t D(\beta_{ij})$ and an $n$-periodic tree $\cT$ so that $\ell=(p_i,p_j)$ is an edge of $\cT$ and $y'$ lies in the face of $F\cR(\cT)$ corresponding to $\ell$.

We start by choosing a point $y''\in U\cap E_\varepsilon^t D(\beta_{ij})$ in the interior of $E_\varepsilon^t D(\beta_{ij})$ so that the second stability condition is strict for all $\beta\subsetneq \beta_{ij}$. Since $F$ is surjective, $y''=F(\pi)$ for some $n$-periodic function $\pi:\ZZ\to \RR$. Any $\pi=\pi_v$ for $v\in D(\beta_{ij})$ must satisfy the conditions of Lemma \ref{subroot lemma}:
\begin{enumerate}
\item $\pi(i)=\pi(j)$.
\item For any $i<k<j$ we have:
	\begin{enumerate}
	\item $\pi(k)>\pi(i)$ if $\varepsilon_k=-1$
	\item $\pi(k)<\pi(i)$ if $\varepsilon_k=1$ 
	\end{enumerate}
\end{enumerate}
Let $m_1,\cdots,m_{n-2}$ be integers which are not congruent to $i$ or $j$ or to each other modulo $n$. Choose $\pi'$ close to $\pi$ so that $\pi'$ satisfies the conditions above and the additional condition that the $n-1$ real numbers $\pi'(m_1),\cdots,\pi'(m_{n-2}),\pi'(i)=\pi'(j)$ are linearly independent over $\QQ$. Then the slope $s$ of $\pi'$ is nonzero and any interval of length $|sn|$ contains at most $n$ values of $\pi'$. Therefore, the distance between consecutive elements in the image of $\pi'$ is bounded below by, say $\delta$. For any real number $t$, let $\pi_t:\ZZ\to\RR$ be the function given by
\[
	\pi_t(k)=\begin{cases} \pi'(k)+t & \text{if } k\equiv j \text{ mod }n\\
   \pi'(k) & \text{otherwise}
    \end{cases}
\]
Then $\pi_t$ is injective for any nonzero $t$ with $|t|<\delta$. Furthermore, for such values of $t$, $\pi_t(i),\pi_t(j)$ will be consecutive values of $\pi_t$ and $\pi_t$ will satisfy Condition (2) above. By Theorem \ref{thm:every pi gives a unique T}, there is a unique $\cT_t$ which clearly depends only on the sign of $t$ so that $\psi(p_k)=\pi_t(k)$ is a periodic morphism for $\cT_t$. By Corollary \ref{converse of monotonic cor}, $\ell(p_i,p_j)$ is an edge of $\cT_t$ with sign equal to the sign of $t$.
\end{proof}

Now we come to the main theorem of this section which is that there is a 1-1 correspondence between $n$-periodic trees with sign function $\varepsilon$ and cluster tilting objects in the cluster category of $\kk\widetilde{A}_{n-1}^\varepsilon$. We first recall definitions.

Recall \cite{BMRRT} that the \emph{cluster category} of $\Lambda=\kk\widetilde{A}_{n-1}^\varepsilon$ is a triangulated Krull-Schmidt category whose indecomposable objects are either  indecomposable $\Lambda$-modules or indecomposable shifted projective $\Lambda$-modules. A \emph{cluster tilting object} is a rigid object which has a maximal number of nonisomorphic direct summands $M_i$. Each summand $M_i$ is \emph{exceptional} (indecomposable and rigid). This is equivalent to $\underline\dim M_i$ being either a positive real Schur root or negative projective root (and $M_i$ being the unique rigid representation of that dimension vector). The summands $M_i$ form a maximal collection of compatible exceptional objects where $M_i,M_j$ are \emph{compatible} if:
\begin{enumerate}
\item $M_i,M_j$ are modules which do not extend each other or
\item $M_i,M_j$ are any two shifted projective modules or
\item One of the objects is a shifted projective $P[1]$ and the other is a module $M$ so that $\Hom_\Lambda(P,M)=0$.
\end{enumerate}
In all three cases we have: $\left<\underline\dim\,M_i,\underline\dim\,M_j\right>\ge0$. We use the notation $|M_i|$ for the underlying module of $M_i$ i.e., $|M|=M$ if $M$ is a module and $|P[1]|=P$.

\begin{lem}\label{null root in generic decomp}
Every $v\in\ZZ^n$ has a generic decomposition
\[
	v=\sum \beta_i
\]
where $\beta_i$ are Schur roots so that $ext(\beta_i,\beta_j)=0$ for $\beta_i\neq\beta_j$. Furthermore, if one of the $\beta_i$ is a null root, then $\left<v,\eta\right>=0$.
\end{lem}

\begin{proof}
Since $\left<\eta,\alpha\right>=-1$ for all preprojective roots $\alpha$, $ext(\eta,\alpha)>0$. Also, $\left<\beta,\eta\right>=-1$ for all preinjective roots $\beta$. So, $ext(\beta,\eta)>0$. The only roots compatible with null roots are the regular Schur roots $\gamma$ with $\left<\gamma,\eta\right>=0$. If one of the objects in the generic decomposition of $v$ is a null root, then $v$ is a sum of regular roots and null roots and $\left<v,\eta\right>=0$.
\end{proof}

The following characterization of cluster tilting objects in terms of semi-invariants and canonical decompositions of general representations is essentially proved in \cite{IOTW}.

\begin{thm}\label{virtual stability theorem}
Suppose that $M_1,\cdots,M_n$ are indecomposable virtual representations of $\widetilde{A}_{n-1}^\varepsilon$ corresponding to exceptional objects in the cluster category of $\kk\widetilde{A}_{n-1}^\varepsilon$. Then the following numbered conditions are equivalent.
\begin{enumerate}
\item $M_1\oplus\cdots\oplus M_n$ is a cluster tilting object in the cluster category.
\item (virtual canonical decomposition theorem) $\underline\dim M_i$ are linearly independent and, for any $v\in \ZZ^n$ which is a nonnegative rational linear combination of the vectors $\underline\dim M_i$, the general virtual representation with dimension vector $v$ is isomorphic to a direct sum of the virtual representations $M_i$.
\item (virtual stability theorem) $\underline\dim M_i$ are linearly independent and the following hold.
	\begin{enumerate}
	\item For each $j$, the set of nonnegative real linear combinations of $\underline\dim M_i$ for $i\neq j$ is contained in the support $D(\beta_j)$ for a uniquely determined real Schur root $\beta_j$.
	\item The set of all $\sum a_i \underline\dim M_i$ where $a_i>0$ for all $i$ is disjoint from $D(\beta)$ for all real Schur roots $\beta$.
	\item If $v=\sum a_i \underline\dim M_i$, $a_i\ge0$, then $\left<v,v\right>\ge0$ and equality holds only when $a_i=0$ for all $i$.
	\end{enumerate}
\end{enumerate}
\end{thm}

\begin{proof}
The equivalence $(1)\ifff (2)$ is proved in \cite{IOTW}. The equivalence with (3) is not too difficult but we did not state this in \cite{IOTW}. The proof is based on ideas in \cite{ST}

$(1),(2)\then (3)$ Suppose that $M_1,\cdots,M_n$ form a cluster tilting object. Then, by a result of Schofield we can arrange the objects so the underlying modules $|M_i|$ form an exceptional sequence (with shifted projective objects $M_i$ moved to the right end and replaced by the projective $|M_i|$). Using braid moves, we can move $M_j$ to the left end and we have a new exceptional sequence: $B_j,|M_1|,\cdots,\widehat{|M_j|},\cdots,|M_n|$ where $B_j$ is an exceptional module with the property that $\Hom(|M_i|,B_j)=0=\Ext(|M_i|,B_j)$ for all $i\neq j$. Equivalently, $\underline\dim M_i\in D_\ZZ(\beta_j)$ where $\beta_j=\underline\dim B_j$. Then (5a) is satisfied. (Need the easy lemma that a projective root $\pi_i$ lies in $D(\beta)$ if and only $\left<\pi_i,\beta\right>=0$ if and only if $-\pi_i\in D(\beta)$.)

To verify (3b), suppose not. Then there is a vector $v$ with integer coefficients which is a positive linear combination of $\underline\dim M_i$ and so that $v\in D_\ZZ(\beta)$ where $\beta$ is the dimension vector of an exceptional object $B$. But this implies that $\Hom(V,B)=0=\Ext^1(V,B)$ for the general virtual representations $V$ of dimension $v$. By assumption, $V$ is a direct sum of copies of the objects $M_i$. So, we must have $\underline\dim M_i\in D_\ZZ(\beta)$ for all $i$. But this is impossible since the vectors $\underline\dim M_i$ span $\RR^n$.

(3c) follows from (1) since $\left<\underline\dim\,M_i,\underline\dim\,M_j\right>\ge0$ for all $i,j$ and $\left<\underline\dim\,M_i,\underline\dim\,M_i\right>=1$ making $\left<v,v\right>\ge \sum a_i>0$ if $a_i$ are not all zero.

Conversely, $(3)\then (1)$. Let $v=\sum\underline\dim\,M_i$ and let $N=\bigoplus b_kN_k$ be the generic decomposition of the general virtual representation with dimension vector $v$.

\underline{Case 1:} All $N_k$ are rigid. ($\underline\dim\,N_k$ are real Schur roots.) 

Then we can extend the set $\{N_k\}$ to a cluster tilting object and $v$ lies in the positive cone $C\Delta_N$ of the $n-1$ simplex spanned by the dimension vectors of the $N_k$. ($C\Delta_N$ is the set of all nonnegative linear combinations of the vectors $\underline\dim\,N_k$.) By assumption (3b) on $\{M_i\}$, the boundary of $C\Delta_N$ does not meets the interior of $C\Delta_M$, the corresponding set for $M$. So $C\Delta_M\subseteq C\Delta_N$. Similarly, $C\Delta_N\subseteq C\Delta_M$. This implies $C\Delta_M= C\Delta_N$. So, $\bigoplus M_i=\bigoplus N_k$ is a cluster tilting object, proving (1)

\underline{Case 2:} At least one of the $N_k$ is not rigid. (So, $\gamma_k=\underline\dim\,N_k$ is an imaginary root.)

We claim that this case is not possible. We prove this by induction on $m$ where $m$ is minimal so that a positive linear combination $v=\sum a_i\beta_i$ of $m$ of the roots $\alpha_i=\underline\dim\,M_i$ contains an imaginary root in its canonical decomposition: $N=\bigoplus b_kN_k$. Let $\Delta_N$ the the $m-1$ simplex spanned by $\underline\dim\,N_k$. Then $v\in C\Delta_M$ where $\Delta_M$ is the $m-1$ simplex spanned by the dimension vectors of all $M_i$ so that $a_i\neq 0$. By assumption (3a), this implies that $C\Delta_M$ is contained in the intersection $L$ of all $D(\beta_j)$ for all $j$ so that $a_j=0$. So, $v$ lies in this intersection. Since this is an open condition, each $\underline\dim\,N_k$ also lies in $L$. By induction on $m$, the interior of $\Delta_N$ does not meet the boundary of $C\Delta_M$. This implies that $\Delta_N\subset C\Delta_M$. In particular $\underline\dim\,N_k\in C\Delta_M$. But this contradicts (3c). So, Case 2 is not possible and $(3)\then(1)$ in both cases.
\end{proof}

\begin{cor}\label{cor: V is invertible}
For any cluster tilting object $M=\bigoplus M_i$, the $n\times n$ integer matrix $V$ whose columns are the dimension vectors $\ul\dim M_i$ has determinant $\pm1$.
\end{cor}

\begin{proof} By (3) in Theorem \ref{virtual stability theorem}, $V$ is invertible as a matrix over $\QQ$. Let $B$ be an integer larger than the absolute value of any entry $a_{ij}$ of $V^{-1}$. By (2) in Theorem \ref{virtual stability theorem} we see that, for each $j$, the integer vector $\sum_i(B+a_{ij})\ul\dim M_i=e_j+B\sum\ul\dim M_i$ is an integer linear combination of the integer vectors $\ul\dim M_i$. Thus, each $a_{ij}$ must be an integer and $V$ is invertible as an integer matrix.
\end{proof}

\begin{thm}\label{thm:correspondence between trees and cluster tilting objects}
There is a 1-1 correspondence between $n$-periodic trees $\cT$ with sign function $\varepsilon$ and cluster tilting objects $M=\bigoplus M_i$ in the cluster category of $\Lambda=\kk\widetilde{A}_{n-1}^\varepsilon$ given by the equation
\begin{equation}\label{cluster region equality}
	F\cR(\cT)=E_\varepsilon^t\cR(M)
\end{equation}
where $\cR(M)\subseteq \RR^n$ is the set of all positive real linear combinations of the vectors $\undim M_i$. Furthermore,
\begin{equation}\label{Hugh Thomas equation}
V^tE_\varepsilon \Gamma_\cT=I_n
\end{equation}
where $V$ is the $n\times n$ matrix whose $k$th column is $\underline\dim M_k$ and $\Gamma_\cT$ is the edge matrix of $\cT$.
\end{thm}

\begin{rem}\label{rem: edge matrix of initial cluster tilting object} For example, let $\Lambda[1]=P_1[1]\oplus \cdots\oplus P_n[1]$ be the cluster tilting object of shifted projective modules. The matrix is $V=(-E_\varepsilon^t)^{-1}$. So, \eqref{Hugh Thomas equation} implies that $\Gamma_\cT=-I_n$, making the edge vectors equal to $-\beta_{i,i+1}$ and $\cT$ is a straight line with slope $-1$. We denote this tree $\cT_0$.
\end{rem}

\begin{proof} Given a cluster tilting object $\bigoplus M_i$, $\cR(M)$ is a nonempty open subset of $\RR^n$. For a general point $v\in \cR(\cT)$ the corresponding $n$-periodic function $\pi_v:\ZZ\to\RR$ (given by $E_\varepsilon^tv=F(\pi)$) takes distinct values on all integers. So, there is a unique $n$-periodic tree $\cT$ so that $\pi_v\in \cR(\cT)$. This implies that $F\cR(\cT)$ and $E_\varepsilon^t\cR(M)$ have a nonempty intersection. However, the boundaries of both sets are contained in the union of the supports $E_\varepsilon^tD(\beta)$. And the sets $E_\varepsilon^tD(\beta)$ do not meet the interior of either set by Lemma \ref{second stability lemma} and Theorem \ref{virtual stability theorem}. Therefore, $F\cR(\cT)=E_\varepsilon^t\cR(M)$.

Conversely, let $\cT$ be an $n$-periodic tree. Then $\cR(\cT)$, being an open set contains a rational point $\pi:\ZZ\to\QQ$ with nonzero slope $\frac mn$. Multiplying by the common denominator we may assume that $\pi$ takes integer values. Then the corresponding dimension vector $v=(E_\varepsilon^t)^{-1}F(\pi)$ has a generic virtual decomposition $v=\sum \beta_i$ which does not contain a null root by Lemma \ref{null root in generic decomp} since $\left<v,\eta\right>=\pi(n)-\pi(0)=\frac mn\neq0$ by construction. Therefore, the roots $\beta_i$ correspond to modules or shifted projective modules which don't extend each other. We can extend this to a cluster tilting object $\bigoplus M_i$ so that $v$ lies in $\overline \cR(M)$. By a small pertubation of $v/||v||$, we may assume that $v\in \cR(M)$. Then \eqref{cluster region equality} holds and furthermore gives a bijection between periodic trees and cluster tilting objects.

Having established the correspondence between $n$-periodic trees and cluster tilting objects, we will now prove the equation \eqref{Hugh Thomas equation}. Given a cluster tilting object $\bigoplus M_i$, let $\beta_i$ be the real Schur roots given in Theorem \ref{virtual stability theorem}. Then
\[
	\left<\ul\dim M_i,\beta_j\right>=0
\]
for $i\neq j$. Then $d_k=\left<\ul\dim M_k,\beta_k\right>$ cannot be zero since the Euler-Ringel form is nondegenerate and $\dim M_k$ span $\RR^n$. So we obtain the matrix equation:
\[
	V^tE_\varepsilon B=D
\]
where $V$ is the $n\times n$ matrix whose $i$th column is $\undim M_i$, $B$ is the integer matrix whose columns are $\beta_i$ and $D$ is a diagonal matrix with diagonal entries $d_k$. By comparing the two descriptions of the set $F\cR(\cT)=E_\varepsilon^t\cR(M)$ given in Lemma \ref{first stability lemma} and Theorem \ref{virtual stability theorem}, we see that the real Schur roots $\beta_k$, are up to sign, equal to the edge vectors of $\cT$. This gives another equation:
\[
	V^tE_\varepsilon \Gamma_\cT=D'
\]
where $D'$ is the diagonal matrix with diagonal entries $|d_k|$. The sign is defined in such a way that, for any $n$-periodic function $\psi$ on $\cT$ and any real Schur root $\beta_k=\beta_{ij}$ of $\cT$, $\left<v,\delta_k\beta_{ij}\right>=\delta_k(\psi(p_j)-\psi(p_i))$ is positive. So, the entries of $D'$ are positive integers.

Finally, we know, by Proposition \ref{sign coherence of c-inverse} and Corollary \ref{cor: V is invertible} that the matrices $V,\Gamma_\cT$ have determinant $\pm1$. So, $D'=I_n$ is the identity matrix as claimed.
\end{proof}

\subsection{Formula for summands of $M$}\label{ss2.4: psi-infty}

The following corollary shows how the geometry of the period tree relates to the summands of the corresponding cluster tilting object. Recall that every periodic tree has a unique periodic infinite path. If this path is monotonically increasing/decreasing the periodic tree has positive/negative slope according to the Classification Theorem \ref{cor: classification of trees}.

Let $\cT$ be an $n$-periodic tree and let $M=\bigoplus M_i$ be the corresponding cluster tilting object. For each summand $M_i$ of $M$, let $\ell_i$ be the corresponding edge of the tree $\cT$. Then the dimension vector $\undim M_i$ can be computed as follows. 

\begin{defn}\label{def: psi infty}
Let $\psi_\infty^i:\{p_k\,:\,k\in\ZZ\}\to\RR$ denote any $n$-periodic function satisfying the following two conditions.
\begin{enumerate}
\item $\psi_\infty^i$ takes the same value at the endpoints of every edge $\ell_j$ not equal to $\ell_i$ or its translates, i.e., the edges $\ell_j$ of $\cT$ become horizontal.
\item $\psi_\infty^i(p_b)=\psi_\infty^i(p_a)+1$ if $\ell_i$ has endpoints $p_a,p_b$ with $p_a<p_b$.
\end{enumerate}
\end{defn}

Then we get the following formula.

\begin{cor}\label{cor: components of cluster from topology of tree} 
Let $\psi_\infty^i$ be given as above. Then $F\psi_\infty^i=E_\varepsilon^t\undim M_i$. The sign of the slope of $\psi_\infty^i$ is equal to the sign of $\left<\undim M_i,\eta\right>$. Thus:
\begin{enumerate}
\item $M_i$ is regular if and only if $\ell_i$ does not lie on the periodic infinite path in $\cT$.
\item $M_i$ is preprojective if and only if $\ell_i$ lies on the periodic infinite path of $\cT$ and either $\cT$ has positive slope or $\cT$ has zero slope and $\ell_i$ has positive slope.
\item $M_i$ is preinjective or shifted projective if and only if $\ell_i$ lies on the periodic infinite path of $\cT$ and either $\cT$ has negative slope or $\cT$ has zero slope and $\ell_i$ has negative slope.
\end{enumerate}
\end{cor}

\begin{proof}
Up to a positive scalar multiple, the dimension vector of $M_i$ is given by taking the limit of $n$-periodic morphisms when the slopes of all edges except for $\ell_i$ and its translates become zero and applying the linear map $(E_\varepsilon^t)^{-1}F$. Thus $\psi_\infty^i$ is this limiting periodic morphism. By \eqref{cluster region equality}, $F(\psi_\infty^i)$ is proportional to $E_\varepsilon^t\undim M_i$. By condition (2) in the definition, $F\psi_\infty^i$ has integer coordinates and $F\psi_\infty^i\cdot \beta_{ab}=1$. So, $F(\psi_\infty^i)=E_\varepsilon^t\undim M_i$. 

Regular, preprojective and preinjective roots can be distinguished by the sign of $\left<\beta,\eta\right>=\eta^tE_\varepsilon^t\beta$ which is equal to the sign of $\eta^tF(\psi_\infty^i)=\psi_\infty^i(p_n)-\psi_\infty^i(p_0)=m$ which is equal to the sign of the slope of $\psi_\infty^i$.

In Case (1) when $\ell_i$ and its translates lie on the branches of $\cT$, all edges in the infinite path become horizontal in the limit and $\psi_\infty^i$ takes the same value at all point in this infinite path. So, its slope is zero and $M_i$ is regular.

In Case (2) when $\ell_i$ is part of the periodic infinite path in $\cT$, suppose that either $\cT$ has zero slope and $\ell_i$ has positive slope or that $\cT$ has positive slope. Then the complement of $\ell_i$ and its translates in $\cT$ is an infinite union of finite trees each of which becomes horizontal by $\psi_\infty^i$. In both subcases of Case (2), the edges $\ell_{i+kn}$ make the height of each finite tree greater than the previous one making $\psi_\infty^i$ to have positive slope.

Similarly, $\psi_\infty^i$ has negative slope in Case (3). So, the correspondence is accurate.
\end{proof}

This leads to the following characterization of periodic trees of positive, negative and zero slope.

\begin{cor}\label{cor: objects corresponding to trees with different slopes}
Let $\cT$ be an $n$-periodic tree and let $M=\bigoplus M_i$ be the corresponding cluster tilting object.
\begin{enumerate}
\item $\cT$ is a zero slope tree if and only if $M$ contains at least one preprojective summand and at least one summand which is either preinjective or negative shifted projective.
\item $\cT$ is a positive slope tree if and only if $M$ has no preinjective or negative shifted projective summands.
\item $\cT$ is a negative slope tree if and only if $M$ has no preprojective summands.\qed
\end{enumerate}
\end{cor}

\begin{cor}\label{cor: bijection respects mutation}
The bijection between periodic trees and cluster tilting objects commutes with mutation, i.e., if $\cT$ corresponds to $M=\bigoplus M_i$ then $\mu_k\cT$ corresponds to the cluster tilting object $\mu_kM$ uniquely determined by the formula $\mu_kM=M/M_k\oplus M_k'$ where $M_k'\not\approx M_k$. 
\end{cor}

\begin{proof} For each $i\neq k$, the function $\psi_\infty^i$ is the same for both $\cT$ and $\mu_k\cT$. The reason is that $\psi_\infty^i$ is given by collapsing all the edges of $\cT$ other than $\ell_i$. But, when $\ell_k\neq\ell_i$ is collapsed, $\cT,\mu_k\cT$ become equal to the same tree $\overline \cT$ used in the proof of Proposition \ref{prop: mutation of tree is tree}. 

By Corollary \ref{cor: components of cluster from topology of tree}, $\psi_\infty^i$ determines the $i$th summand $M_i$ of the cluster tilting object. Therefore, the cluster tilting objects corresponding to $\cT$ and $\mu_k\cT$ differ only in their $k$th summands. So, they are mutations of each other in the $k$th direction.
\end{proof}


\subsection{Example}\label{ss2.5: example}
We illustrate the main Theorem \ref{thm:correspondence between trees and cluster tilting objects} and its Corollary \ref{cor: components of cluster from topology of tree} on the $4$-periodic tree given in Figure \ref{fig0a}. The sign function is $-,+,+,+$. So, the Euler matrix and its inverse are:
\[
	E_\varepsilon=\mat{1 & 0 & 0 & 0\\
	-1 & 1 & 0 & 0\\
	0 & -1 & 1 & 0\\
	-1 & 0 & -1 & 1}
	,\quad
	E_\varepsilon^{-1}=\mat{1 & 0 & 0 & 0\\
	1 & 1 & 0 & 0\\
	1 & 1 & 1 & 0\\
	2 & 1 & 1 & 1}.
\]
The rows of $E_\varepsilon^{-1}$ are the dimension vectors of the projective modules. The edges in the tree are $(p_0,p_2),-(p_2,p_4),(p_3,p_4),-(p_3,p_5)$. So
\[
	\Gamma=\mat{1 & 0 & 0 & -1\\
	1 & 0 & 0 & 0\\
	0 & -1 & 0 & 0\\
	0 & -1 & 1 & -1}	,\quad
	E_\varepsilon \Gamma=\mat{1 & 0 & 0 & -1\\
	0 & 0 & 0 & 1\\
	-1 & -1 & 0 & 0\\
	-1 & 0 & 1 & 0}.
\]
The periodic tree $\cT$ corresponds to a cluster tilting object $M=\bigoplus M_i$ whose components $M_i$ are given as follows. The computational formula \eqref{Hugh Thomas equation} gives the dimension vectors of the $M_i$ as the rows of $(E_\varepsilon \Gamma)^{-1}$. The geometric formula for $M_i$, given in Corollary \ref{cor: components of cluster from topology of tree}, is as follows.
\begin{enumerate}
\item $M_1$ is preprojective since it corresponds to the edge $\ell_1=(p_0,p_2)$ which has positive slope and is part of the infinite path of this zero slope tree. The dimension vector of $M_1$ is given by ``flattening'' the other edges to give:
\[
\xymatrixrowsep{10pt}\xymatrixcolsep{10pt}
\xymatrix{
&&&&&&&&&p_6\ar@{-}[r] &p_7\ar@{-}[r] &p_8 \ar@{-}[r]& \\
\psi_\infty^1: &&&&&p_2\ar@{-}[r] & p_3\ar@{-}[r] & p_4\ar@{-}[rru]^{\ell_1^+}\ar@{-}[r] &p_5\\
&\,\ar@{-}[r] & p_{-1}\ar@{-}[r] & p_0\ar@{-}[r]\ar@{-}[rru]^{\ell_1} &p_1
} 
\]
Then, $F(\psi_\infty^1)=(0,1,0,0)^t$. So, $M_1=P_2$ with $(\undim M_1)^t=(1,1,0,0)$.
\item $M_2$ is either preinjective or shifted projective according to Corollary \ref{cor: components of cluster from topology of tree}(3).
\[
\xymatrixrowsep{10pt}\xymatrixcolsep{10pt}
\xymatrix{
&& p_{-1}\ar@{-}[r] & p_0\ar@{-}[r] &p_1\ar@{-}[r]&p_2 \ar@{-}[rrd]^{\ell_2} &\\
\psi_\infty^2: &&&&& & p_3\ar@{-}[r] & p_4 \ar@{-}[r] &p_5 \ar@{-}[r] & p_6 \ar@{-}[rrd]^{\ell_2^+}& \\
&&&&&&&&& &p_7\ar@{-}[r] &p_8\ar@{-}[r] & \, 
} 
\]
Then, $F(\psi_\infty^2)=(0,0,-1,0)^t$. So, $M_2=P_3[1]$ with $(\undim M_2)^t=(0,0,-1,-1)$.
\item $M_3$ is regular since $\ell_3$ lies on a branch of $\cT$.
\[
\xymatrixrowsep{1pt}\xymatrixcolsep{10pt}
\xymatrix{
 &\,\ar@{-}[rr]&&p_0\ar@{-}[rr] &&p_2 \ar@{-}[rr]&  & p_4 \ar@{-}[rr] && p_6 \ar@{-}[rr] && \\
\psi_\infty^3: \\
&&p_{-1}\ar@{-}[ruu]^(.4){\ell_3}\ar@{-}[rr]&&p_1&&p_3\ar@{-}[rr]\ar@{-}[ruu]^(.4){\ell_3^+}&&p_5& &
} 
\]
Making $F(\psi_\infty^3)=(-1,1,-1,1)^t$. So, there is an exact sequence:
\[
	0\to P_1\oplus P_3\to P_2\oplus P_4\to M_3\to 0
\]
and $(\undim M_3)^t=(1,1,0,1)$.
\item $M_4$ is also regular.
\[
\xymatrixrowsep{1pt}\xymatrixcolsep{10pt}
\xymatrix{
 &\,\ar@{-}[r]&p_{-1}\ar@{-}[r]&p_0\ar@{-}[rr] &&p_2 \ar@{-}[r]& p_3 \ar@{-}[r]& p_4 \ar@{-}[rr] && p_6 \ar@{-}[rr] && \\
\psi_\infty^4: \\
&&&&p_1\ar@{-}[uull]\ar@{-}[uull]^{\ell_4}&&&&p_5\ar@{-}[uull]^{\ell_4^+}& &
} 
\]
Making $F(\psi_\infty^4)=(-1,1,0,0)^t$. This means there is an exact sequence:
\[
	0\to P_1\to P_2\to M_4\to 0
\]
So, $M_4=S_2$ is simple with $(\undim M_4)^t=(0,1,0,0)$.
\end{enumerate}
Putting these together we get:
\[
\mat{ (\undim M_1)^t\\
(\undim M_2)^t\\
(\undim M_3)^t\\
(\undim M_4)^t} = 
\mat{ 1 & 1 & 0 & 0\\
-1 & -1 & -1 & 0\\
1 & 1 & 0 & 1\\
0 & 1 & 0 & 0}=(E_\varepsilon\Gamma)^{-1}
\]
as claimed by the first formula.

\begin{eg}
One more example: Consider the edge $\ell_1$ in Figure \ref{fig1}
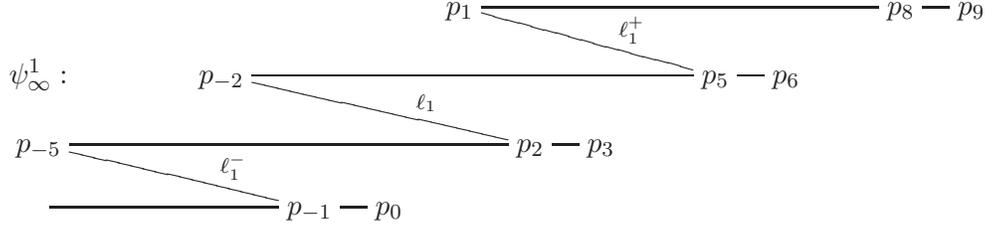
\[
\xymatrixrowsep{10pt}\xymatrixcolsep{10pt}
\xymatrix{
&&&&&&p_1\ar@{-}[rrrrrrr]&&&&&& &p_8 \ar@{-}[r]& p_9\\
\psi_\infty^1:&&&p_{-2}\ar@{-}[rrrrrrr] &  & & &  &  &  & p_5\ar@{-}[r]\ar@{-}[ullll]_(.35){\ell_1^+} & p_6\\
p_{-5}\ar@{-}[rrrrrrr] &&& &  &  &  & p_2\ar@{-}[r]\ar@{-}[ullll]_(.35){\ell_1} & p_3\\
\,\ar@{-}[rrrr] & &  &  & p_{-1}\ar@{-}[r]\ar@{-}[ullll]_(.3){\ell_1^-} & p_0
} 
\]
Then $F(\psi_\infty^1)=(3,-2,0)^t$. So, there is an exact sequence:
\[
	0\to P_2\oplus P_2\to P_1\oplus P_1\oplus P_1\to M_1\to 0
\]
making $M_1$ the preprojective module with $(\undim M_1)^t=(3,3,4)$.
\end{eg}


\section{Edge vectors are negative $c$-vectors}\label{sec3}

In Section \ref{sec3} we review the definition of the $c$-vectors of a cluster tilting object and show that the edge vectors of a periodic tree are equal to the negatives of the $c$-vectors of the corresponding cluster tilting object. 



\subsection{Exchange matrix and cluster tilting objects}\label{ss3.1: exchange matrix and cluster tilting objects}

We review the definition of the exchange matrix of a cluster tilting object in the simply laced case. 

\begin{defn}\label{def: exchange matrix of a cluster tilting object}
Let $Q$ be a quiver without oriented cycles and let M=$\bigoplus M_i$ be a cluster tilting object in the cluster category of $\kk Q$. Then the \emph{exchange matrix} $B=B_M=(b_{ij})$ of the cluster tilting object is defined to be the skew-symmetric integer matrix given by
\[
	b_{ij}=\dim_\kk \underline\Hom(M_j,M_i)-\dim_\kk \underline\Hom(M_i,M_j)
\]
where $\underline\Hom(M_i,M_j)$ is the quotient of $\Hom_{\cD^b}(M_i,M_j)$ by the subspace of all morphisms $M_i\to M_j$ in the cluster category which factors through some $M_k$ where $k\neq i,j$.
\end{defn}

For example, take the cluster tilting object $\Lambda[1]$ whose components are the shifted projective objects $P_1[1],\cdots,P_n[1]$. By Remark \ref{rem: edge matrix of initial cluster tilting object}, this is the cluster tilting object which corresponds to the straight line tree $\cT_0$ with edge vectors $-\beta_{i-1,i}$ and edge matrix equal to the negative identity matrix $\Gamma_{\cT_0}=-I_n$. The exchange matrix is $B_{\Lambda[1]}=E_\varepsilon^t-E_\varepsilon$. We call this the \emph{initial exchange matrix}.

Suppose that $M_k'$ is the unique object of the cluster category not isomorphic to $M_k$ so that $\mu_k(M):=M\backslash M_k\cup M_k'$ is a cluster tilting object. Then the basic theorem of cluster theory is:

\begin{thm}\label{thm: exchange matrix transforms according to FZ}
The exchange matrix of $\mu_k(M)$ is the matrix $B'=(b_{ij}')$ given as follows.
\begin{enumerate}
\item $b_{ij}'=-b_{ij}$ if either $i=k$ or $j=k$.
\item $b_{ij}'=b_{ij}+b_{ik}|b_{kj}|$ if $i,j\neq k$ and $b_{ik},b_{kj}$ have the same sign.
\end{enumerate}
\end{thm}

\begin{rem} With the notation $B'=\mu_k(B)$, the theorem says: $B_{\mu_k(M)}=\mu_k(B_M)$. 
\end{rem}



\subsection{Statement of the theorem}\label{ss3.2: statement of the theorem}

We can now give several equivalent formulations of the theorem that edge vectors are negative $c$-vectors.

\begin{defn}
Given an $n$-periodic tree $\cT$ with sign function $\varepsilon$, let $M=\bigoplus M_i$ be the corresponding cluster tilting object in the cluster category of $\kk \widetilde{A}_{n-1}^\varepsilon$ and let $\Gamma_\cT$ be the {edge matrix} of $\cT$ with columns in the corresponding order. We define the \emph{extended exchange matrix} $\widetilde B$ of $\cT$ and $M$ to be the $2n\times n$ matrix 
\[
	\widetilde B:=\mat{B_M\\ -\Gamma_\cT}
\]
where $B_M$ is given in Definition \ref{def: exchange matrix of a cluster tilting object} above.
\end{defn}

For example, if $\cT=\cT_0$ and $M=P$, we have the \emph{initial extended exchange matrix}
\[
	\widetilde B_0=\mat{E_\varepsilon^t-E_\varepsilon\\ I_n}.
\]
The main theorem about edge vectors and cluster tilting objects is the following.

\begin{thm}\label{thm: mutation of extended exchange matrix}
Under mutation of the tree $\cT$ and corresponding mutation of the cluster tilting object $M$ (Corollary \ref{cor: bijection respects mutation}), the extended exchange matrix mutates by the same rules as in Theorem \ref{thm: exchange matrix transforms according to FZ}. In other words:
\[
	\mu_k(\widetilde B)=\mat{B_{\mu_kM}\\ -\Gamma_{\mu_k\cT}}.
\]
\end{thm}

Since cluster mutation acts transitively on the set of all cluster tilting objects, this theorem implies and in fact is equivalent to the following.

\begin{thm}\label{thm: edge vectors are negative c-vectors}
The edge vectors of a periodic tree $\cT$ are equal to the negatives of the $c$-vectors of the corresponding cluster tilting object $M$.
\end{thm}

By a theorem of Nakanishi and Zelevinsky \cite{NZ} and the fact that edge vectors are sign coherent, these statements are equivalent to the following.

\begin{thm}\label{thm: NZ formula for B}
The exchange matrix $B_M$ of a cluster tilting object $M$ is related to the edge matrix $\Gamma_\cT$ of the corresponding $n$-periodic tree $\cT$ by the following formula.
\[
	B_M=\Gamma_\cT^t (E_\varepsilon^t-E_\varepsilon)\Gamma_\cT.
\]
In other words,
\[
	b_{ij}=\left<\gamma_j,\gamma_i\right>-\left<\gamma_i,\gamma_j\right>.
\]
\end{thm}

When we compute the numbers $\left<\gamma_i,\gamma_j\right>$ for the edge vectors of a periodic tree $\cT$, we will be able to compute the exchange matrix $B_M$ and thereby obtain the quiver $Q_M$ of the corresponding cluster tilting object $M$.

\begin{thm}\label{thm: formula for QM}
The quiver $Q_M$ of the cluster tilting object $M$ corresponding to a periodic tree $\cT$ is dual to the tree in the following sense.
\begin{enumerate}
\item The quiver $Q_M$ has one vertex $v_i$ for every edge vector $\gamma_i=\beta_{pq}$.
\item Two vertices of $Q_M$ are connected by one or two arrows $v_i\to v_j$ if the corresponding edges meet at one or two vertices of $\cT$ respectively.
\item The orientation of the arrow $v_i\to v_j$ is always counterclockwise around each vertex of $\cT$. (See Figure \ref{fig2}.)
\end{enumerate}
\end{thm}

%
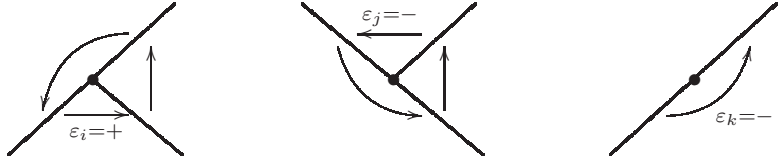
\begin{figure}[htbp]
\begin{center}
%
{
\setlength{\unitlength}{1cm}
{\mbox{
\begin{picture}(10,2.5)
      \thicklines
\put(0.1,.5){
\qbezier(-.1,0)(1,1)(2.1,2)
\qbezier(1,1)(1.6,0.5)(2.2,0)
\put(.92,.9){$\bullet$}
}
\put(0,2){
$
\xymatrix{
& \quad \ar@/_1pc/[dl]\\
\quad \ar[r]_{\varepsilon_i=+} & \quad \ar[u]
}
$
} 
\put(4.1,.5){
\qbezier(1,1)(1.55,1.5)(2.1,2)
\qbezier(-.2,2)(1,1)(2.2,0)
\put(.92,.9){$\bullet$}
}
\put(3.9,2){
$
\xymatrix{
 \quad \ar@/_1pc/[dr] & \quad \ar[l]_{\varepsilon_j=-}\\
& \quad \ar[u]
}
$
} 
  \put(8,0){\put(0.1,.5){
\qbezier(-.1,0)(1,1)(2.1,2)
\put(.92,.9){$\bullet$}
}
\put(0,2){
$
\xymatrix{
& \quad \\
\quad \ar@/_1pc/[ru]_{\varepsilon_k=-}
}
$
} }
\end{picture}}
}}
\caption{Orientation of $Q_M$ around $p_i$ where $\varepsilon_i=+$: Parent $\to$ left Child $\to$ right Child $\to$ Parent. Around $p_j$ where $\varepsilon_j=-$, the arrows are oriented: Child $\to$ right Parent $\to$ left Parent $\to$ Child. In the third figure, there is only one arrow Child $\to$ right Parent since $\varepsilon_k=-$. (The absent left Parent blocks any arrow going the other way.)}
\label{fig2}
\end{center}
\end{figure}

As an example of the theorem we construct the quiver $Q_M$ of the quiver corresponding to the periodic tree in Figure \ref{fig1}. Here the three edges $\ell_1,\ell_2,\ell_3$ meet at one vertex $p_5$ and are ordered counterclockwise around that vertex. So, the corresponding edges of the quiver $Q_M$ form a triangle oriented: $v_1\to v_2\to v_3\to v_1$. The edges $\ell_1,\ell_2$ also meet at vertex $p_1$ in $\cT$ where, again, $\ell_1$ is clockwise from $\ell_2$. So, there are two arrows $v_1\to v_2$. Therefore, the exchange matrix is
\[
	B_M=\mat{0 & 2 & -1\\
	-2 & 0 & 1\\
	1 & -1 & 0}
\]
and the quiver is:
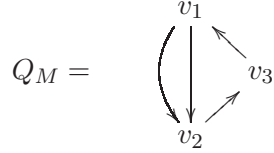
\[
\xymatrixrowsep{10pt}\xymatrixcolsep{10pt}
\xymatrix{
&&v_1 \ar[dd]\ar@/_1pc/[dd]\\
Q_M=&& & v_3\ar[ul]\\
& &v_2\ar[ru]
	}
\]

\subsection{Outline of the proof of the theorems}\label{ss3.3: outline of proof}

These theorems follows from three propositions. The first, \ref{first prop}, gives the calculation of $\left<\gamma_j,\gamma_k\right>$ for all pairs of edge vectors of any periodic tree $\cT$. The second, \ref{second prop}, gives a description of mutation of a periodic tree in terms of its edge vectors. 

Define the ``candidate exchange matrix'' to be the $n\times n$ matrix $\Gamma_\cT^t(E_\varepsilon^t-E_\varepsilon)\Gamma_\cT$ with entries
\[
	b_{ij}=\left<\gamma_j,\gamma_i\right>-\left<\gamma_i,\gamma_j\right>.
\]
The first two propositions give a formula for how this matrix changes under mutation of periodic trees. The third proposition \ref{third prop} states that this mutation formula agrees with the formula given in Theorem \ref{thm: exchange matrix transforms according to FZ}. Since the initial value of the candidate exchange matrix is equal to $B_{\Lambda[1]}$, we conclude that the candidate exchange matrix is equal to the exchange matrix in all cases, proving Theorems \ref{thm: NZ formula for B} and \ref{thm: exchange matrix transforms according to FZ}.

The third proposition gives slightly more. If we define the ``extended candidate exchange matrix'' of $\cT$ to be the $2n\times n$ matrix
\[
	\mat{
	\Gamma_\cT^t(E_\varepsilon^t-E_\varepsilon)\Gamma_\cT\\
	-\Gamma_\cT
	}
\]
the third proposition states that this larger matrix transforms according to the Fomin-Zelevinsky mutation rules given in Theorem \ref{thm: exchange matrix transforms according to FZ}. This proves Theorem \ref{thm: mutation of extended exchange matrix}.

Finally, the first proposition, giving the values of $\left<\gamma_j,\gamma_k\right>$, implies Theorem \ref{thm: formula for QM} since we now know that $Q_M=\Gamma_\cT^t(E_\varepsilon^t-E_\varepsilon)\Gamma_\cT$. This proves all versions of the result.


\subsection{First proposition}\label{ss3.4: prop 1}

Suppose that $\gamma_a,\gamma_b$ are edge vectors of an $n$-periodic tree $\cT$. Then, there are three possibilities. Either the edges are disjoint, they share one endpoints or they share two endpoints.

\begin{lem}\label{lem 1a}
If $\gamma_a,\gamma_b$ are edge vectors of $\cT$ which correspond to disjoint edges of $\cT$ then $\left<\gamma_a,\gamma_b\right>=0$.
\end{lem}

Thus, the candidate exchange matrix has a zero as $(a,b)$-entry if the corresponding edges $\ell_a,\ell_b$ are disjoint. A useful lemma in the calculation is:

\begin{lem}
Given two roots $\alpha,\beta$ of $\widetilde{A}_{n-1}^\varepsilon$, let $\tilde\beta$ be any fixed lifting of $\beta$ to the (infinite) universal covering quiver $\widetilde{A_\varepsilon}$ of $\widetilde{A}_{n-1}^\varepsilon$. Then,
\begin{equation}\label{eq: pairing is sum over coverings}
	\left<\alpha,\beta\right>=\sum \left<\tilde\alpha,\tilde\beta\right>
\end{equation}
where the sum is over all liftings $\tilde \alpha$ of $\alpha$ to $\widetilde{A_\varepsilon}$.
\end{lem}
\begin{proof}
Since Equation \eqref{eq: pairing is sum over coverings} is linear in $\alpha,\beta$, it suffices to show that it holds for simple roots. But this case is clear.
\end{proof}

We will use the notation $\tilde\ell$ to refer to the translate of $\ell$ corresponding to $\tilde\gamma$.

\begin{proof}[Proof of Lemma \ref{lem 1a}]
We use the covering formula \eqref{eq: pairing is sum over coverings} and show that every term in this formula is zero. So, let $\tilde\gamma_a=\pm\beta_{ij},\tilde\gamma_b=\pm\beta_{k\ell}$ be two edge vectors corresponding to disjoint edges $\tilde \ell_a,\tilde\ell_b$ of the infinite tree $\cT$. By vertical symmetry, there are three cases as indicated below.
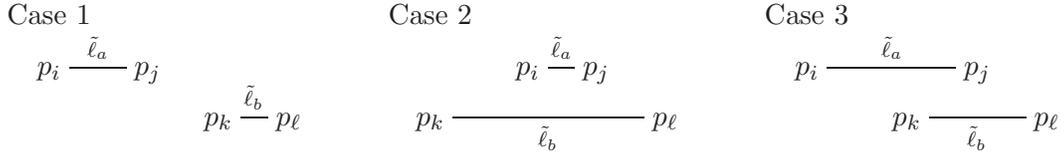
\[
\xymatrixrowsep{5pt}\xymatrixcolsep{10pt}
\xymatrix{
	  \text{Case 1} &&&& & \text{Case 2} &&&&& \text{Case 3}\\
p_i\ar@{-}[r]^{\tilde\ell_a} &
	p_j & & && & p_i\ar@{-}[r]^{\tilde\ell_a}& p_j &&& p_i\ar@{-}[rr]^{\tilde\ell_a}&& p_j \\
	&& p_k \ar@{-}[r]^{\tilde\ell_b}& p_\ell & &p_k\ar@{-}[rrr]_{\tilde\ell_b}&&& p_\ell &&&p_k\ar@{-}[rr]_{\tilde\ell_b}&& p_\ell  
	}
\]
\begin{enumerate}
\item If $i<j<k<\ell$ then clearly $\left<\beta_{ij},\beta_{k\ell}\right>=\left<\beta_{k\ell},\beta_{ij}\right>=0$.
\item If $k<i<j<\ell$ then $\varepsilon_i=\varepsilon_j=-$. So, $ext(\beta_{ij},\beta_{ki})=0$, $ext(\beta_{ij},\beta_{j\ell})=1$ and
\[
	\left<\beta_{ij},\beta_{k\ell}\right>=\left<\beta_{ij},\beta_{ki}\right>+\left<\beta_{ij},\beta_{ij}\right>+\left<\beta_{ij},\beta_{j\ell}\right>=0+1-1=0.
\]
Similarly, $\left<\beta_{k\ell},\beta_{ij}\right>=0$.
\item If $i<k<j<\ell$ then $\varepsilon_k=+,\varepsilon_j=-$.  So, $ext(\beta_{ik},\beta_{k\ell})=0$, $ext(\beta_{kj},\beta_{j\ell})=1$ and
\[
	\left<\beta_{ij},\beta_{k\ell}\right>=\left<\beta_{ik},\beta_{k\ell}\right>+\left<\beta_{kj},\beta_{kj}\right>+\left<\beta_{kj},\beta_{j\ell}\right>=0+1-1=0.
\]
Similarly, $\left<\beta_{k\ell},\beta_{ij}\right>=0$.
\end{enumerate}
So,, $\left<\beta_{ij},\beta_{k\ell}\right>=0$ in $\widetilde{A_\varepsilon}$ in all cases where $i,j,k,\ell$ are distinct making $\left<\gamma_a,\gamma_b\right>=0$.
\end{proof}

If $\ell_a,\ell_b$ share one endpoint, there are three cases. Either they have the same left endpoint, they have the same right endpoint or the endpoint they share is the right endpoint of one and the left endpoint of the other. If $\ell_a,\ell_b$ share a left endpoint, say $p_i$, then one must be ascending from $p_i$ and one must be descending. 

\begin{rem}\label{rem: assume n is large}
By the proof of the previous lemma, we may assume that $n$ is much bigger than the lengths of the edges of $\cT$. (In the covering formula, $\left<\tilde\gamma_a,\tilde\gamma_b\right>=0$ unless one of the endpoints of $\tilde\ell_a$ is equal to one of the endpoints of $\tilde\ell_b$. So, all terms in \eqref{eq: pairing is sum over coverings} are zero except for the ones which look like the case under discussion without any ``wrapping around'', i.e., we can ignore the possibility that the edges have length greater than $n/3$.)
\end{rem}

\begin{lem}\label{lem 1b}
Suppose that $\gamma_a,\gamma_b$ are edge vectors of $\cT$ and either $\gamma_a=\beta_{ij}$ and $\gamma_b=-\beta_{ik}$ or $\gamma_b=-\beta_{ij}$ and $\gamma_a=\beta_{kj}$ (so that $\ell_b$ is clockwise from $\ell_a$). Then, $\left<\gamma_a,\gamma_b\right>=0$ and $\left<\gamma_b,\gamma_a\right>=-1$.
\end{lem}

\begin{proof} Take the first case $\gamma_a=\beta_{ij}$ and $\gamma_b=-\beta_{ik}$. Then there are two subcases: $i<j<k$ or $i<k<j$. Take the first. By Remark \ref{rem: assume n is large}, we may assume $k-i<<n$. Then $\varepsilon_j=-$. So, $ext(\beta_{ij},\beta_{jk})=1$ and $hom(\beta_{ij},\beta_{ij})=1$ making $\left<\beta_{ij},\beta_{ik}\right>=0$ and $hom(\beta_{ik},\beta_{ij})=1$ so $\left<\beta_{ik},\beta_{ij}\right>=1$. A similar calculation gives the same result in all four subcases, namely: $hom(|\gamma_a|,|\gamma_b|)=ext(|\gamma_a|,|\gamma_b|)=0$ making $\left<\gamma_a,\gamma_b\right>=0$ and $hom(|\gamma_b|,|\gamma_a|)=1$ making $\left<\gamma_b,\gamma_a\right>=-1$ as claimed.
\end{proof}

\begin{lem}\label{lem 1c}
Suppose that $\gamma_a,\gamma_b$ are edge vectors of $\cT$ and $\gamma_a=\pm\beta_{ij}$, $\gamma_b=\pm\beta_{jk}$ where $k-i$ is not divisible by $n$. Then 
\[
	\left<\gamma_a,\gamma_b\right>-\left<\gamma_b,\gamma_a\right>=(\sgn\gamma_a)(\sgn\gamma_b)\varepsilon_j.
\]
\end{lem}

\begin{proof}
By Remark \ref{rem: assume n is large}, we may assume $k-i<<n$. So, $\beta_{ij},\beta_{jk}$ are hom-orthogonal. Then the formula follows from the observation that, if $\varepsilon_j=+$, then $ext(\beta_{ij},\beta_{jk})=0$ and $ext(\beta_{jk},\beta_{ij})=1$ and, if $\varepsilon_j=-$, then $ext(\beta_{ij},\beta_{jk})=1$ and $ext(\beta_{jk},\beta_{ij})=0$.
\end{proof}

Finally, it can happen that two edges share both endpoints, as in Figure \ref{fig1}.

\begin{lem}\label{lem 1d}
Suppose that $\gamma_a,\gamma_b$ are edge vectors of $\cT$ and $\gamma_a=\pm\beta_{ij}$, $\gamma_b=\pm\beta_{jk}$ where $k-i$ is divisible by $n$. Then 
\[
	\left<\gamma_a,\gamma_b\right>-\left<\gamma_b,\gamma_a\right>=(\sgn\gamma_a)(\sgn\gamma_b)\varepsilon_j-(\sgn\gamma_a)(\sgn\gamma_b)\varepsilon_i\neq0.
\]
In other words, the two endpoints $p_i,p_j$ give separate contributions to $\left<\gamma_a,\gamma_b\right>-\left<\gamma_b,\gamma_a\right>$ following Lemma \ref{lem 1c} and these contributions never cancel each other.
\end{lem}

\begin{proof}
Use the covering formula. There are only two terms which are nonzero: the term where $\tilde\ell_a$ is adjacent to $\tilde\ell_b$ from the left and the term where it is adjacent to $\tilde\ell_b$ from the right.  By the previous lemma, each contributes a separate summand as indicated. 

It remains to show that $\varepsilon_i\neq\varepsilon_j$. To see this, note that the two edges are connected end to end. So, the translates of the two edges give an infinite curve dividing the plane in half. If both signs were, say, negative then the rest of the tree must be above this infinite curve and therefore all signs of all vertices must be negative. But this is excluded by assumption since it corresponds to the case when the quiver $\widetilde{A}_{n-1}^\varepsilon$ has an oriented cycle. 
\end{proof}

These lemmas together can be summarized as follows.

\begin{prop}\label{first prop}
Suppose that $\gamma_a,\gamma_b$ are edge vectors of an $n$-periodic tree $\cT$. Then $\left<\gamma_b,\gamma_a\right>-\left<\gamma_a,\gamma_b\right>$ is equal, in absolute value, to the number of endpoints that the corresponding two edges $\ell_a,\ell_b$ share. The sign of this quantity is positive if and only if $\ell_b$ is counterclockwise from $\ell_a$ at each vertex that they share.
\end{prop}



\subsection{Second proposition}\label{ss3.5: prop 2}

Let $\cT$ be an $n$-periodic tree and let $\gamma_k=\beta_{ab}$ be a positive edge vector of $\cT$. Let $\cT'=\mu_k(\cT)$ be the mutation of $\cT$ in the $k$th direction. We give a formula for the edge vectors of the mutated tree $\cT'$.

\begin{prop}\label{second prop}
For every edge vector $\gamma_j$ of $\cT$ there is a corresponding edge vector $\gamma_j'$ of $\cT'$ given as follows.
\begin{enumerate}
\item $\gamma_k'=-\gamma_k$.
\item $\gamma_j'=\gamma_j+\gamma_k$ if $\ell_j$ connects $p_b$ to $p_c$ which is a left/only parent of $p_b$ for $\varepsilon_b=-/+$.
\item $\gamma_j'=\gamma_j+\gamma_k$ if $\ell_j$ connects $p_a$ to $p_d$ which is a right/only child of $p_a$ for $\varepsilon_a=+/-$.
\item $\gamma_j'=\gamma_j+2\gamma_k$ if $\ell_j$ connects $p_b$ to a translate $p_{a+sn}$ of $p_a$ so that $p_{a+sn}$ is a left/only parent of $p_b$ for $\varepsilon_b=-/+$, respectively.
\item $\gamma_j'=\gamma_j$ in all other cases.
\end{enumerate}
\end{prop}

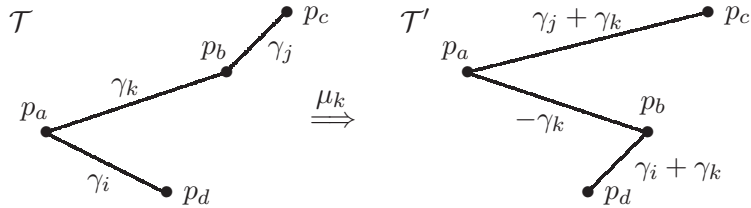
\begin{figure}[htbp]
\begin{center}
%
{
\setlength{\unitlength}{.8cm}
{\mbox{
\begin{picture}(11,3)
      \thicklines
  %
     \put(0,0){ 
\put(-.1,-.1){  \put(2,0){$\bullet\ p_d$} \put(-.5,2.8){$\cT$}  \put(0,1){$\bullet$}
    \put(3,2){$\bullet$}
    \put(0.8,.2){$\gamma_i$}
    }
\put(-.4,0){     \put(0,1.3){$p_a$}
    \put(3,2.3){$p_b$}
    \put(4.1,2.3){$\gamma_j$}
    \put(1.5,1.7){$\gamma_k$}
    }
      \qbezier(0,1)(1,.5)(2,0) 
      \qbezier(0,1)(1.5,1.5)(3,2) 
    \qbezier(3,2)(3.5,2.5)(4,3) 
    \put(3.9,2.9){$\bullet\ p_c$}
    }
    \put(7,0){ 
\put(-.1,-.1){ \put(2,0){$\bullet\ p_d$} \put(-1,2.8){$\cT'$}  \put(0,2){$\bullet$}
    \put(3,1){$\bullet$}}
\put(-.4,0){     \put(0,2.3){$p_a$}
    \put(3.3,1.4){$p_b$}
        \put(1.5,2.8){$\gamma_j+\gamma_k$}
    \put(1.2,1.1){$-\gamma_k$}
        \put(3.2,0.3){$\gamma_i+\gamma_k$}
}
      \qbezier(3,1)(2.5,.5)(2,0) 
      \qbezier(0,2)(1.5,1.5)(3,1) 
    \qbezier(0,2)(2,2.5)(4,3) 
        \put(3.9,2.9){$\bullet\ p_c$}
    }
    \put(4.5,1.5){$\mu_k$}
    \put(4.4,1.1){$\Longrightarrow$}
\end{picture}}
}}
\caption{Additive formula for $\cT'=\mu_k\cT$: Right child of $p_a$ slides over to $p_b$ and $\gamma_i$ becomes $\gamma_i+\gamma_k$. If $\varepsilon_b=+$, the only parent of $p_b$ slides over to $p_a$ and $\gamma_j$ becomes $\gamma_j+\gamma_k$.}
\label{fig3}
\end{center}
\end{figure}

\begin{proof}
When the slope of the $k$th edge changes from positive in $\cT$ to negative in $\cT'$, $p_a$ will become a left parent of $p_b$ and $p_b$ will become a right child of $p_a$. By Definition \ref{def: mutation of T}, any right or only child of $p_a$ in $\cT$ becomes a new child of $p_b$ in $\cT'$ and any left or only parent of $p_b$ in $\cT$ becomes a new parent of $p_a$ in $\cT'$. Each of these moves will add the edge vector $\gamma_k=\beta_{ab}$ to the edge vector which is being modified. In case (4) translates of the edge $\ell_k$ are added to both ends of $\ell_j$ and $\gamma_j'=\gamma_j+2\gamma_k$.
\end{proof}


\subsection{Third proposition}\label{ss3.6: prop 3}


We come to the final proposition which proves the theorem that edge vectors are negative $c$-vectors.

\begin{prop}\label{third prop}
Under mutation of periodic trees, the candidate extended exchange matrix
\[
	\mat{\Gamma_\cT^t (E_\varepsilon^t-E_\varepsilon)\Gamma_\cT\\
	-\Gamma_\cT
}
\]
transforms according to the Fomin-Zelevinsky rules (as given in the statement of Theorem \ref{thm: exchange matrix transforms according to FZ}).
\end{prop}

Since the matrix has two parts, the proof is in two parts describing the mutation of $\Gamma_\cT$ and the resulting mutation of $B=\Gamma_\cT^t (E_\varepsilon^t-E_\varepsilon)\Gamma_\cT$. Let $b_{ij}$ denote the $ij$ entry of this candidate exchange matrix $B$.

\begin{lem}\label{lem: mutation of Gamma}
Under mutation of a periodic tree $\cT$ in the $k$ direction, the edge vectors $\gamma_j$ change to vectors $\gamma_j'$ given as follows.
\begin{enumerate}
\item $\gamma_k'=-\gamma_k$.
\item If $j\neq k$ then
\[
	\gamma_j'=\begin{cases} \gamma_j+|b_{kj}|\gamma_k & \text{if $b_{kj}$, $\gamma_k$ have opposite signs}\\
  \gamma_j  & \text{otherwise}
    \end{cases}
\]
\end{enumerate}
\end{lem}

\begin{proof}[Proof of Lemma \ref{lem: mutation of Gamma}]
We need to verify the formula for $\gamma_j'$. There are several cases.

If $j=k$ then we have $\gamma_k'=-\gamma_k$ by definition of $\mu_k$.

If the $j$th edge $\ell_j$ is disjoint from the $k$th edge $\ell_k$ then $b_{kj}=0$ by Lemma \ref{lem 1a} and $\gamma_j'=\gamma_j$ by Proposition \ref{second prop}. So the formula holds in this case.

Now suppose that $\gamma_k=\beta_{ab}$ and $\ell_j$ shares one endpoint with $\ell_k$. By symmetry we assume it is the right endpoint $p_b$. Then there are five cases summarized by the following chart.
\[
\begin{array}{c|c|c|c|c}
\varepsilon_b & \gamma_j & b_{kj} & \gamma_j' & \text{relation of $\ell_j$ to $p_b$}\\
\hline
- & -\beta_{cb} & -1 & \gamma_j+\gamma_k & \text{left parent}\\
- & \beta_{bc} & +1 & \gamma_j & \text{right parent}\\
+ & -\beta_{bc} & +1 & \gamma_j & \text{right child}\\
+ & -\beta_{cb} & -1 & \gamma_j+\gamma_k & \text{only parent}\\
+ & \beta_{bc} & -1 & \gamma_j+\gamma_k & \text{only parent}\\
\end{array}
\]
In detail: if $\varepsilon_b=-1$ then $\gamma_j$ is either a left parent or right parent of $p_b$. In the first case, $b_{kj}=-1$ by Proposition \ref{first prop} and $\gamma_j'=\gamma_j+\gamma_k$ by Proposition \ref{second prop}(3). In the second case, $b_{kj}=1$ and $\gamma_j'=\gamma_j$. If $\varepsilon_b=+1$ then $\gamma_j$ is either a right child of $p_b$ or the only parent. In the first case, $b_{kj}=1$ and $\gamma_j'=\gamma_j$. In the second case, $b_{kj}=-1$ by Proposition \ref{first prop} and $\gamma_j'=\gamma_j+\gamma_k$ by Proposition \ref{second prop}. So, the formula holds in both cases.

Finally, suppose that $\ell_j$ shares both of its endpoints with $\ell_k$. Then there are four possibilities as outlined in the following table.
\[
\begin{array}{c|c|c|c|c|c}
\varepsilon_a & \varepsilon_b & \gamma_j & b_{kj} & \gamma_j' & \text{relation of $\ell_j$ to $p_a$ and $p_b$}\\
\hline
- & + & \beta_{b,a+sn} & -2 & \gamma_j+2\gamma_k & \text{only child of $p_a$, only parent of $p_b$}\\
\pm &\mp & -\beta_{a,b+sn} (s\neq0) & -2 & \gamma_j+2\gamma_k & \text{right child of $p_a$, left parent of $p_b$}\\
- &+ & -\beta_{b,a+sn} & +2 & \gamma_j & \text{left parent of $p_a$, right child of $p_b$}\\
+ &-& \beta_{b,a+sn} & +2 & \gamma_j & \text{left child of $p_a$, right parent of $p_b$}\\
\end{array}
\]
The second item in this table is illustrated in Figure \ref{fig1}. The formula for $b_{kj}=\left<\gamma_j,\gamma_k\right>-\left<\gamma_k,\gamma_j\right>$ is given by Proposition \ref{first prop}. The formula for $\gamma_j'$ is given by Proposition \ref{second prop}.
\end{proof}

By Nakanishi and Zelevinsky, this lemma (together with the sign coherence of the vector $\gamma_k$) implies that $B=\Gamma_\cT^t (E_\varepsilon^t-E_\varepsilon)\Gamma_\cT$ mutates correctly and is therefore equal to the exchange matrix. The details are given as follows.

\begin{lem}\cite{NZ}
If $\Gamma'$ is the mutation of $\Gamma$ as given in the previous lemma, then the entries $b_{ij}'$ of the matrix $B'=(\Gamma_\cT')^t (E_\varepsilon^t-E_\varepsilon)\Gamma_\cT'$ are given by
\begin{enumerate}
\item $b_{ij}'=-b_{ij}$ if $i=k$ or $j=k$.
\item $b_{ij}'=b_{ij}+|b_{ik}|b_{kj}$ if $i,j\neq k$ and $b_{ik}b_{kj}\ge0$
\item $b_{ij}'=b_{ij}$ otherwise.
\end{enumerate}
\end{lem}

\begin{proof}
Suppose $i$ or $j$ is equal to $k$. Say, $i=k$. Let
\[
	c_j:=\begin{cases} |b_{kj}| & \text{if $b_{kj},\gamma_k$ have opposite sign}\\
   0 & \text{otherwise}
    \end{cases}
\]
so that $\gamma_j'=\gamma_j+c_j\gamma_k$ for all $j\neq k$. Then, we get:
\[
	b_{kj}'=\left<\gamma_j+c_j\gamma_k,-\gamma_k,\right>-\left<-\gamma_k,\gamma_j+c_j\gamma_k\right>=-b_{kj}+(-c_j+c_j)\left<\gamma_k,\gamma_k\right>=-b_{kj}.
\]

 For $i,j\neq k$ we get
 \begin{eqnarray*}
 	b_{ij}'&=&\left<\gamma_j',\gamma_i'\right>-\left<\gamma_i',\gamma_j'\right>\\
	&=& \left<\gamma_j+c_j\gamma_k,\gamma_i+c_i\gamma_k\right>-\left<\gamma_i+c_i\gamma_k,\gamma_j+c_j\gamma_k\right>
	\\
	&=& b_{ij}+c_j\left(
	\left<\gamma_k,\gamma_i\right>-
	\left<\gamma_i,\gamma_k\right>
	\right)
	+c_i \left(
	\left<\gamma_j,\gamma_k\right>-\left<\gamma_k,\gamma_j\right>
	\right)\\
	&=& b_{ij}+c_j b_{ik}+c_ib_{kj}
\end{eqnarray*}

There are two cases. If $b_{ki},b_{kj}$ have the same sign (or one is zero) then either $c_i,c_j$ are both zero or $c_i=|b_{ki}|$ and $c_j=|b_{kj}|$ in which case $c_j b_{ik}+c_ib_{kj}=-|b_{ki}| b_{ki}+|b_{ki}|b_{kj}=0$. So, both subcases give $b_{ij}'=b_{ij}$. 

The second case is when $b_{ki},b_{kj}$ are nonzero with opposite signs. Say, $b_{kj}\gamma_k$ is positive and $b_{ki}\gamma_k$ is negative. Then $c_i=|b_{ki}|$ and $c_j=0$ making $b_{ij}'=b_{ij}+ |b_{ki}|b_{kj}$ as claimed. The other subcase is similar.
\end{proof}

This concludes the proof that the edge vectors of an $n$-periodic tree are the negatives of the $c$-vectors of the corresponding cluster tilting object.

%
%


\end{document}

%